\documentclass{amsart}
\usepackage{amsmath, amssymb}
\usepackage[T1, OT4]{fontenc}
\usepackage{fullpage}
\usepackage[utf8]{inputenc}
\usepackage{graphicx}
\usepackage{mathrsfs}
\usepackage{euscript}
\usepackage{enumerate}
\usepackage{accents}
\usepackage{cancel}
\usepackage{color}
\usepackage{faktor}
\usepackage{enumitem}
\usepackage{caption, subcaption}
\usepackage{float}
\usepackage{yhmath}
\usepackage{hyperref}
\usepackage{stackrel}
\usepackage{cleveref}

\hypersetup
{
    colorlinks = true,
    linkcolor = black,
    filecolor = black,      
    urlcolor = black,
    citecolor = black,
}

\makeatletter
\newcommand\setItemnumber[1]{\setcounter{enum\romannumeral\@enumdepth}{\numexpr#1-1\relax}}
\makeatother

\newcommand{\OPN}[1]{\operatorname{#1}}
\newcommand{\row}[2]{#1_1, \ \ldots , \ #1_{#2}}
\newcommand{\rowto}[2]{#1 = 1, \ \ldots , \ #2}
\newcommand{\rowft}[1]{0, \ \ldots , \ #1}

\newcommand{\puo}[2]{#1^{(#2)}}

\newcommand*{\defeq}{\stackrel{{\OPN{def}}}{=}}

\theoremstyle{plain}
\newtheorem{thm}{Theorem}[section]
\newtheorem{lem}[thm]{Lemma}
\newtheorem{prop}[thm]{Proposition}
\newtheorem{cor}[thm]{Corollary}
\newtheorem{prob}[thm]{Problem}

\theoremstyle{definition}
\newtheorem{deff}[thm]{Definition}

\newtheorem*{note}{Note}
\newtheorem{ex}[thm]{Example}

\newtheorem{fa}[thm]{Fact}

\makeatletter\@namedef{subjclassname@2020}{\textup{2020} Mathematics Subject Classification} \makeatother

\title{On inductive and inverse limits of systems\\ of compact metric spaces\\ and applications}
\author[K. Urba\'{s}]{Kamil Urba\'{s}}
\address{
Jagiellonian University\\{}Institute of Mathematics\\{}
6 Stanis\l{}awa \L{}ojasiewicza Street\\{}30-348 Krak\'{o}w\\{}Poland}
\email{kamil.urbas@im.uj.edu.pl}

\begin{document}

\begin{abstract}
The aim of this paper is to prove the existence of inductive and inverse limits of direct and inverse systems in a certain category of compact metric spaces as well as of compact metric groups.
Some applications are presented.
\end{abstract}

\thanks{This work was supported by the National Center of Science [grant 2021/03/Y/ST1/00072]}

\maketitle

\section{Introduction}

The direct system and its limit are the basic concepts of Category Theory that are also applicable in~Topology.
Nevertheless, in case of nearly all categories in which objects are topological spaces, the~inductive limit of~a~direct system (if it exists) composed of compact spaces, in general, is not a compact space.
For~this~very~reason, any general theorem on the compactness of the limit of a direct system seems to be especially valuable and interesting.
In this paper we present one of such theorems.
More precisely, we prove that in the category of all compact metric spaces with non--expansive surjections as arrows, all~direct systems have inductive limits (see Theorem \ref{main}).
Especially surprising is the fact that the object of the limit, up~to~isometry, does not depend on~the~choice of~the~arrows of the direct system, but only on the objects that appear in it.
As we will see in~Section~\ref{Section 2}, the~result mentioned above is a consequence of the Gromov's Theorem on the existence of~a~universal compact metric space for families of such spaces (see \cite{Gromov}).
Thanks to this theorem, we also prove the theorem on existence of inverse limit in the same category (see Theorem~\ref{inverse limit - main}).
This time the object of the limit also turns out, up to isometry, not to depend on the choice of the arrows of the system.
However, and it is worth emphasizing, not every inverse system has inverse limit in this category, as it is with the category of compact spaces with continuous mappings as arrows.

The aforementioned Theorems \ref{main}, \ref{inverse limit - full version} turn out to have natural counterparts in the category of compact metric groups with bi--invariant metrics (see Theorems \ref{main_groups}, \ref{inverse limit - main_bi_mg}), and in the category of compact metric groups with left--invariant metrics (see Theorems \ref{main_inductive_compact_metric_groups}, \ref{main_inverse_compact_metric_groups}).
The proofs of these theorems in case of categories with groups with bi--invariant metrics, admittedly, are based on their previously mentioned basic versions, however they are more subtle - it takes some effort to get a group structure on the boundary objects.
For~groups with left--invariant metrics, the situation becomes even more complicated (see Example \ref{ex left-inv}).

We also present some application of Theorem \ref{main} to the construction of so-called \emph{iso--derivatives} of compact metric spaces.
For each such space, we define its \emph{iso--height}, analogously to the case of the theory of Cantor--Bendixson of scattered spaces.
In point (\ref{trivial Iso group}) of Proposition \ref{prop before def of iht} and in Theorem \ref{all of countable heights} we show that the iso--height can be any countable ordinal and only such.
Details on this can be found in Section \ref{Section 5}.

We believe, that our new theorem on the existence of an inductive limit will find wider applications in~the~future.
A few potential suggestions for using it are outlined in the last, \ref{Section 6}'th Section.

\subsection*{Notation and terminology} \( \)\\
\indent To avoid any confusion, we introduce some agreements concerning few notions which we will use in~this~paper.
First of all, we make an agreement that \(0 \notin \mathbb{N}\) and \(\overline{\mathbb{N}} = \mathbb{N} \cup \{\infty\}\).
Denote \(\mathbb{R}_{+} = [0\), \(\infty)\). \linebreak
If \(A\) and \(B\) are non--empty subsets of a compact metric space \((X\), \(d)\), then the \emph{distance} between them is the~value 
\[
\displaystyle \OPN{dist}_d(A, \, B) \defeq \inf\{d(a, \, b) \colon \ a \in A, \, b \in B\}.
\]
Similarly, the diameter of the set \(A\) is denoted by
\[
\OPN{diam}_d(A) \defeq \sup \{d(a_1, \, a_2) \colon \ a_1, \, a_2 \in A\}.
\]
We will say that \(A\) is \emph{\(\varepsilon\)--dense} in \(X\), for some positive \(\varepsilon\), if 
for every \(x \in X\) one can find \(a \in A\) such~that \(d(x\),~\(a) < \varepsilon\).

The main results of this paper are on the existence of the inductive and inverse limits of direct and inverse systems in~some categories.
Therefore, it will not hurt to recall what do we mean by a direct system~and an~inverse system.
Let then \(\Sigma\) be an upward directed set.
By a \emph{direct system} in a category \(\textsf{A}\) we call a~family of~objects \(\{A_\sigma\}_{\sigma \in \Sigma}\) in \(\textsf{A}\) together with a family \(\{\alpha_{\sigma}^{\tau} \colon A_\sigma \longrightarrow A_\tau\}_{\sigma \leqslant \tau}\) of morphisms in~\(\textsf{A}\), such~that \(\alpha_{\sigma}^{\sigma} = \OPN{id}_{A_\sigma}\) 
and \(\alpha_{\varrho}^{\tau} \circ \alpha_{\sigma}^{\varrho} = \alpha_{\sigma}^{\tau}\), where \(\sigma\), \(\varrho\) and \(\tau \in \Sigma\) are such that \(\sigma \leqslant \varrho \leqslant \tau\).
If \(\Sigma = \mathbb{N}\), one uses simplified notation: \(\{\alpha_n \colon A_n \longrightarrow A_{n + 1}\}_{n \in \mathbb{N}}\) (and then \(\alpha_{m}^{n} = \alpha_{n - 1} \circ \ldots \circ \alpha_{m}\) for \(m < n\)).
The system defined in this way will be denoted briefly as~a~pair \((\{A_\sigma\}_{\sigma \in \Sigma}\),~\(\{\alpha_{\sigma}^{\tau} \colon A_\sigma \longrightarrow A_\tau\}_{\sigma \leqslant \tau})\).
As for an inverse system, since it is~a~concept dual to a direct system, we define it symmetrically to that above.
In fact, to~define~it, one~must reverse the direction of all of the above arrows.
To be more explicit, an \textit{inverse system} in~a~category \(\textsf{A}\) is~a~family of~objects \(\{A^\sigma\}_{\sigma \in \Sigma}\) in \(\textsf{A}\) together with a family \(\{\alpha_{\tau}^{\sigma} \colon A^\tau \longrightarrow A^\sigma\}_{\sigma \leqslant \tau}\) of morphisms in~\(\textsf{A}\), such~that \(\alpha_{\sigma}^{\sigma} = \OPN{id}_{A^\sigma}\) 
and \(\alpha_{\varrho}^{\sigma} \circ \alpha_{\tau}^{\varrho} = \alpha_{\tau}^{\sigma}\), where \(\sigma\), \(\varrho\) and \(\tau \in \Sigma\) are such that \(\sigma \leqslant \varrho \leqslant \tau\).
Similarly as~before, the system defined in this way will be denoted briefly as~a~pair \((\{A^\sigma\}_{\sigma \in \Sigma}\),~\(\{\alpha_{\tau}^{\sigma} \colon A^\tau \longrightarrow A^\sigma\}_{\sigma \leqslant \tau})\).
In~our~case, we will~be~interested in the category of all compact metric spaces with non--expansive mappings as arrows.
We denote this category by \(\textsf{cmn}\).

We will also prove the~above-mentioned main result in the version for compact metric groups with \linebreak \emph{bi--invariant metrics} (that~is, metrics that are both left-- and right--invariant).
For this purpose, it will~be convenient to agree on what a metric group is.
By a \emph{metric group} we mean a metrizable topological group \(G\) together with a~fixed left--invariant
metric \(d\) inducing its~topology.
Of course, a \emph{metric group with bi--invariant metric} is the metric group for which the~above~metric \(d\) is bi--invariant.
On the basis of~the~Birkhoff--Kakutani Theorem, for there to exist a compatible left--invariant metric on the topological group, it is necessary and sufficient for this group to be first countable and \(T_2\).
In turn, for compact groups, the given condition is~equivalent to the existence of a compatible bi--invariant metric on~this~group, which is an easy and immediate consequence of the aforementioned Birkhoff--Kakutani Theorem.

\subsection*{Gromov--Hausdorff space}\( \)\\
\indent The main background for our considerations will be the Gromov--Hausdorff space introduced in \(1981\) by~Gromov in \cite{Gromov}.
For this reason, in this section we provide a brief introduction to this space.
The~Reader interested in diving deeper into this subject is referred to 
\cite{
Metric Geometry,
Chowdhury and Memoli,
Gromov,
Gromov2,
Gromov3,
Ivanov Iliadis and Tuzhilin - 1,
Ivanov Iliadis and Tuzhilin - 2,
Ivanov Nikolaeva and Tuzhilin - 1,
Ivanov Nikolaeva and Tuzhilin - 2,
Ivanov and Tuzhilin - 1,
Ivanov and Tuzhilin - 2,
Ivanov and Tuzhilin - 3,
Ivanov and Tuzhilin - 4,
Pakhomova,
Petersen,
Tsvetnikov}.

In the same paper, Gromov introduced an intuitive method of measuring distances between metric spaces, being a natural generalisation of Hausdorff distance.

\begin{deff} \label{dGH}
Let \((X\), \(d)\) and \((Y\), \(\varrho)\) be metric spaces.
We define the \emph{Gromov--Hausdorff distance} between these spaces to be the value
\[
\displaystyle d_{GH}((X, \, d), \, (Y, \, \varrho)) \defeq \inf_{(Z, \, \zeta)} \, \inf_{\varphi, \, \psi}\zeta_H(\varphi(X), \, \psi(Y)),
\]
where \((Z\), \(\zeta)\) is a metric space, and \(\varphi \colon (X\), \(d) \longrightarrow (Z\), \(\zeta)\), \(\psi \colon (Y\), \(\varrho) \longrightarrow (Z\), \(\zeta)\) are isometric embeddings, while \(\zeta_H(\varphi(X), \, \psi(Y))\) is the Hausdorff distance between \(\varphi(X)\), \(\psi(Y)\) with respect to \(\zeta\).
\end{deff}

This was the original approach for measuring distances between metric spaces proposed by Gromov.
However, it turns out that in the above formula one can consider only disjoint unions of our metric spaces, instead of all metric spaces containing them.

For the purposes of introducing this equivalent formula for \(d_{GH}\), and the considerations that follow, it~will~be convenient to define so-called \emph{admissible metrics} on countable disjoint unions of metric spaces.

\begin{deff}
Let \(\displaystyle \{(X_n\), \(d_n)\}_{n = 1}^{n = \infty}\) be a sequence of compact metric spaces and fix \(1 < N \leqslant \infty\). \linebreak
By an~\emph{admissible metric} on \(\bigsqcup_{n = 1}^{n = N} X_n\) we mean any metric \(\varrho\) on it such that \(\varrho|_{X_n \times X_n} = d_n\) for~every \nolinebreak \(n \leqslant N\).

Assume that the sequence of diameters of \(X_n\) (where \(n \leqslant N\)) is monotone decreasing, and that there exists a sequence \(\{r_n\}_{n < N}\) 
of positive numbers such that
\(2r_{n - 1} \geqslant r_n\) for every \(1 < n < N\), and \(r_n \geqslant \frac{1}{2}\OPN{diam}_{d_n}X_n\) for every \(n < N\).
\begin{enumerate}
\item An admissible metric \(\varrho\) on \(\bigsqcup_{n = 1}^{N} X_n\) is called to be \emph{with parameter(s)} \(\{r_n\}_{n < N}\) (or \(\row{r}{N - 1}\) for~finite~\(N\) and \(r_1\), \(r_2\), \(\ldots\) for \(N = \infty\)), if \(\varrho(x_k\),~\(x_l) = r_k\) for all \(x_k \in X_k\), \(x_l \in X_l\), where \(k\),~\(l \in \mathbb{N}\) are~such that \(k < l \leqslant N\).
\item Similarly, an admissible metric \(\varrho\) on \(\bigsqcup_{n \in \overline{\mathbb{N}}} X_n\) is called to be \emph{with parameters} \(\{r_n\}_{n \in \mathbb{N}}\) (or \(r_1\), \(r_2\),~\(\ldots\)), if, in addition to the above, \(\varrho(x_k\),~\(x_\infty) = r_k\) for~all \(x_k \in X_k\), \(x_\infty \in X_\infty\), where \(k \in \mathbb{N}\).
\end{enumerate}
\end{deff}

The following is a kind of folklore.

\begin{prop}
If \((X\), \(d)\) and \((Y\), \(\varrho)\) are metric spaces, then
\[
\displaystyle d_{GH}((X, \, d), \, (Y, \, \varrho)) = \inf_{d^{X \sqcup Y}}d_H^{X \sqcup Y} (X, \, Y),
\]
where \(d^{X \sqcup Y}\) is an admissible metric on disjoint union \(X \sqcup Y\).
\end{prop}

\begin{note}
It is worth to highlight that the \(\inf\) in the above proposition (thus, in  Definition \ref{dGH} as~well) is never taken over the empty set, therefore Gromov--Hausdorff distance is well defined.
What is more, if~\(d_{GH}((X\),~\(d)\),~\((Y\),~\(\varrho)) = 0\) for some compact metric spaces \((X\), \(d)\), \((Y\), \(\varrho)\), then \(X\) and \(Y\) are isometric.
\end{note}

We will mostly be using this second formula for Gromov--Hausdorff distance.
We are able now to define the Gromov--Hausdorff space.

\begin{deff}
\emph{Gromov--Hausdorff space} is the metric space \((\mathfrak{GH}\), \(d_{GH})\), where \(\mathfrak{GH}\) is the space of~all~compact metric spaces up to isometry, while \(d_{GH}\) is as above.
\end{deff}

In aforementioned paper, i.e. \cite{Gromov}, Gromov characterized compact subsets of his new space.
(He did it in~an~impressive way, we strongly encourage the reader to read his proof.)
This characterization has found very important applications.
Inter alia, Gromov succeeded to characterize groups of polynomial growth with~its help.
It will also be one of the most important tools for us.

Before stating this theorem, following Gromov, we need to introduce the following concept.

\begin{deff}
We call a family \(\mathfrak{K} \subset \mathfrak{GH}\) \emph{uniformly compact}, if it is uniformly bounded, as well as for~every \(\varepsilon > 0\) there exists \(N \in \mathbb{N}\) such that every element of \(\mathfrak{K}\) contains \(N\)--element \(\varepsilon\)--dense set. 
\end{deff}

\begin{thm}[Gromov, \(1981\)]\( \) \label{Gromov} \\
If \(\mathfrak{K} \subset \mathfrak{GH}\), then the following conditions are equivalent:
\begin{enumerate}
\item \(\mathfrak{K}\) is relatively compact in \(\mathfrak{GH}\),
\item \(\mathfrak{K}\) is uniformly compact,
\item there exists a compact space \((Y\), \(\varrho)\) that contains an isometric copy of every element of \(\mathfrak{K}\).
\end{enumerate}
\end{thm}

We now move to the next chapter in which we prove the main result of this paper.

\section{Compact metric spaces} \label{Section 2}
\subsection{Inductive limit}
Before we state the first theorem regarding categorical limits, it is worth recalling the definition of an \emph{inductive} (or~\emph{direct}) \emph{limit} of a direct system.

\begin{deff} \label{Inductive limit}
Let \((\{A_\sigma\}_{\sigma \in \Sigma}\), \(\{\alpha_{\sigma}^{\tau} \colon A_\sigma \longrightarrow A_\tau\}_{\sigma \leqslant \tau})\) be a direct system in~a~category \(\textsf{A}\).
By its \emph{inductive} (or \emph{direct}) \emph{limit} we mean an object \(A\) in \(\textsf{A}\) together with a family of~morphisms \(\{\alpha_\sigma \colon A_\sigma \longrightarrow~A\}\) in \(\textsf{A}\) such~that
\begin{enumerate}
\item \(\alpha_\tau \circ \alpha_{\sigma}^{\tau} = \alpha_\sigma\) for \(\sigma \leqslant \tau\) in \(\Sigma\), \label{Inductive limit 1}
\end{enumerate}
and the following statement is fulfilled:
\begin{enumerate}
\item[(2)]\label{Inductive limit 2} if \(A'\) is an object in \(\textsf{A}\) and a family \(\{\xi_{\sigma} \colon A_\sigma \longrightarrow A'\}_{\sigma \in \Sigma}\) of morphisms in \(\textsf{A}\) are such~that
\(\xi_\tau \circ \alpha_{\sigma}^{\tau} = \xi_\sigma\) for \(\sigma \leqslant \tau\) in \(\Sigma\),
then there exists a unique morphism \(\nu \colon A \longrightarrow A'\) such that \(\nu \circ \alpha_\sigma = \xi_\sigma\) for any \nolinebreak \(\sigma \in \Sigma\).
\end{enumerate}
\end{deff}

Of course, if the inductive limit exists (which is not always true), its object is unique up~to~isomorphism, as the arrow \(\nu\) from the above definition is to be unique.

As a warm--up, we will first prove the main theorem for direct system indexed by \(\mathbb{N}\).
The proof for~this~case is slightly less complex than the proof for general case, however, unlike the general one, it can be proven constructively.

\begin{thm} \label{main_N}
Let \((\{(X_n\), \(d_n)\}_{n \in \mathbb{N}}\), \(\{g_n \colon X_n \longrightarrow X_{n + 1}\}_{n \in \mathbb{N}})\) be a direct system in 
\(\mathsf{cmn}\), where \(g_n\) are~surjective for all \(n \in \mathbb{N}\).
Then this system has inductive limit.\\
Furthermore, 
\begin{enumerate}
\item its object does not depend, up~to~isometry, on the choice of surjective arrows  for the system, and \label{main_N f1}
\item its arrows are surjective.
\end{enumerate}
\end{thm}

\begin{note}
In our construction we will be passing to a subsequence countably many times.
To avoid making the~notation more cumbersome, our subsequences always keep the same notation as for the original sequence, until the moment, when we need our final subsequence to define a sequence of arrows for the direct limit.
We~are making an agreement that the first element of any subsequence has index \(1\).
\end{note}

\begin{proof}[Proof of Theorem \ref{main_N}]
For simplicity, let \(f_m^n\) stand for \(g_{n - 1} \circ \ldots \circ g_m\) for \(m < n\) and \(f_n^n = \OPN{id}_{X_n}\) for \(n \in \mathbb{N}\). \linebreak
\( \)
\indent First, notice that all the diameters \(\OPN{diam}_{d_n} X_n\) for \(n \in \mathbb{N}\) have a common upper bound.
What is more, for~any \(\varepsilon > 0\) we can find a positive integer \(N\) such that for any \(n \in \mathbb{N}\) the space \(X_n\) has at most \(N\)--element \linebreak \(\varepsilon\)--dense subset.
Indeed, by the fact that arrows of our system are non--expansive and surjective, it is sufficient to find for fixed \(\varepsilon\) a number \(N\) and \(N\)--element \(\varepsilon\)--dense subset for \(X_1\), which is straightforward by~the~compactness of \(X_1\).
We thus infer that \(\{X_n\}_{n \in \mathbb{N}}\) is uniformly compact, so we use Gromov's theorem (Theorem~\ref{Gromov}) to~find a~compact metric space \((Y\), \(\varrho)\) in~which we can isometrically embed all the \(X_n\)'s for~\(n \in \mathbb{N}\).
Therefore, without loss of~generality, we can consider \(X_n\) for \(n \in \mathbb{N}\) as a closed subset of \(Y\). 
By the compactness of~the~hyperspace of a compact metric space, we know that \(\{X_n\}_{n \in \mathbb{N}}\) has a convergent subsequence in the hyperspace of \(Y\).
However, right now we are only interested in finding such a convergent subsequence, so, without loss of generality, for~a~while, we shall assume that the sequence \(\{X_n\}_{n \in \mathbb{N}}\) converges, with respect~to~the~Hausdorff distance, to some compact subset \(X\) of \(Y\).
Without loss of generality we~may~and do assume that \(\OPN{dist}_{\varrho}(X_n\), \(X) < \frac{1}{2n}\) for~\(n > 1\).
We will prove that \(X\) is the object of~the~direct limit we are~looking~for.

Let a sequence \(\{y_k\}_{k \in \mathbb{N}}\) be dense in \(X\).
We know that there is a sequence \(\{x_n''\}_{n \in \mathbb{N}} \in \prod_{n \in \mathbb{N}} X_n\), for which \(\varrho(x_n''\), \(y_1) < \frac{1}{2n}\) for \(n > 1\).
We find \(\{x_n'\}_{n \in \mathbb{N}} \subset X_1\) such that \(f_1^n(x_n') = x_n''\).
Without loss of generality we can assume that \(\varrho(x_n'\), \(x_1) < \frac{1}{2n}\) for some \(x_1 \in X_1\) and \(n > 1\).
By the triangle inequality, and the fact that \(f_1^n\) is non--expansive, we get \(\varrho(f_1^n(x_1)\), \(y_1) < \frac{1}{n}\) for \(n > 1\).
Proceeding in the same way and using the~diagonal method we construct \(\{x_k\}_{k \in \mathbb{N}} \subset X_1\) such that \(\varrho(f_1^n(x_k)\), \(y_k) < \frac{1}{n}\) for \(k \in \mathbb{N}\) and \(n > 1\).
That~is, \(f_1^n(x_k) \longrightarrow y_k\) in \(Y\) when \(n \longrightarrow \infty\) and \(k \in \mathbb{N}\).

Let \(D\) be a countable dense subset of \(X_1\).
If \(D \backslash \{x_k\}_{k \in \mathbb{N}}\) is non--empty, take a one--to--one sequence \(\{z_k\}_{k < K} = D \backslash \{x_k\}_{k \in \mathbb{N}}\), where \(K \leqslant \infty\).
There exists \(\{y_n'\}_{n \in \mathbb{N}} \subset X\) such that \(\varrho(f_1^n(z_1)\), \(y_n') < \frac{1}{2n}\) for \(n > 1\).
Without~loss~of~generality, \(\varrho(y_n'\), \(w_1) < \frac{1}{2n}\) for \(n > 1\) and some \(w_1 \in X\).
Then \(\varrho(f_1^n(z_1)\), \(w_1) < \frac{1}{n}\) for \(n > 1\).
Once again, the diagonal argument provides us a sequence \(\{w_k\}_{k < K} \subset X\) for which \(\varrho(f_1^n(z_k)\), \(w_k) < \frac{1}{n}\) for~\(k < K\) and \(n > 1\).
That~is, \(f_1^n(z_k) \longrightarrow w_k\) in \(Y\) when \(n \longrightarrow \infty\) and \(k < K\).
In case of \(D \backslash \{x_k\}_{k \in \mathbb{N}} = \emptyset\), we do nothing as \(\{x_k\}_{k \in \mathbb{N}}\) is already dense in \(X_1\), so we do not need to enlarge it.

Let \(D' = \{x_k\}_{k \in \mathbb{N}} \sqcup \{z_k\}_{k < K}\).
Define \(f_1 \colon D' \longrightarrow \{y_k\}_{k \in \mathbb{N}} \cup \{w_k\}_{k < K}\) by \(f_1(x_k) = y_k\) for \(k \in~\mathbb{N}\) and \(f_1(z_k) = w_k\) for \(k < K\).
(We assume here that \(D \backslash \{x_k\}_{k \in \mathbb{N}} \neq \emptyset\). Otherwise, here and while defining  arrows for~our inductive limit, we simply omit parts of the definition in which the sequences \(\{z_k\}_{k < K}\) and \(\{w_k\}_{k < K}\) are involved, and the rest of the proof remains the same.)
By the inequalities proven before, \linebreak \(f_1\) is well~defined and non--expansive.
We will now define mappings \(f_j \colon f_{1}^{j}(D') \longrightarrow \{y_k\}_{k \in \mathbb{N}} \cup \{w_k\}_{k < K}\) for~\(j > 1\).
For~such~\(j\), we define \(f_j\) by the rule:
\(f_j(x) \in f_1((f_1^j)^{-1}(\{x\}) \cap D')\) for every \(x \in f_{1}^{j}(D')\).
We will prove that~these~mappings are well defined and non--expansive.
Fix \(\puo{x_1}{j}\), \(\puo{x_2}{j} \in f_{1}^{j}(D')\).
Choose \(\puo{a_k}{j} \in (f_1^j)^{-1}(\{\puo{x_k}{j}\})\) and let \(\puo{y_k}{j} = f_1(\puo{a_k}{j})\) for \(k = 1\), \(2\).
Let \(\{\eta_n\}_{n \in \mathbb{N}}\) be the set of indices of the final subsequence from~our construction.
Once more, by the inequalities that we have proved previously, one can notice that
\[
\displaystyle \puo{y_k}{j}
\displaystyle = \lim_{n \rightarrow \infty} f_{1}^{\eta_n}\left(\puo{a_k}{j}\right)
\displaystyle = \lim_{n \rightarrow \infty} f_{j}^{\eta_n}\left(f_{1}^{j}\left(\puo{a_k}{j}\right)\right)
\displaystyle = \lim_{n \rightarrow \infty} f_{j}^{\eta_n}\left(\puo{x_k}{j}\right),
\]
for \(k = 1\), \(2\).
Consequently, the right hand side of the inequality
\(
\varrho(\puo{x_1}{j}\), \(\puo{x_2}{j}) \geqslant \varrho(f_{j}^{\eta_n}(\puo{x_1}{j})\), \(f_{j}^{\eta_n}(\puo{x_2}{j}))
\)
converges to~\(\varrho(\puo{y_1}{j}\),~\(\puo{y_2}{j})\).
Thus \(\puo{y_1}{j} = \puo{y_2}{j}\) if only \(\puo{x_1}{j} = \puo{x_2}{j}\), so \(f_j\) is well defined and non--expansive.
However, \(f_n\) for \(n \in \mathbb{N}\) is a non--expansive mapping from a dense subset of \(X_n\) on a dense subset of \(X\), so~it~can be extended in a unique way to a non--expansive surjection between \(X_n\) and \(X\), denoted again by~\(f_n\).
Therefore, what we are really dealing with is the~family of surjective arrows \(\{f_n \colon X_n \longrightarrow X\}_{n \in \mathbb{N}}\) in~\(\mathsf{cmn}\).
Moreover, this family clearly satisfies point (\ref{Inductive limit 1}) from~Definition~\ref{Inductive limit}.

To prove the second point from the definition of inductive limit, fix a compact metric space \((Z\), \(\zeta)\) and a~family of arrows \(\{h_n \colon X_n \longrightarrow Z\}_{n \in \mathbb{N}}\) such that \(h_m \circ f_m^n = h_m\) for all natural numbers \(m < n\).
We need to~find an arrow \(\nu \colon X \longrightarrow Z\) such that \(\nu \circ f_n = h_n\) for every \(n \in \mathbb{N}\).
Suppose that we already have an arrow \(\nu\) with slightly weaker property, i.e., \(\nu \circ f_1 = h_1\).
By the surjectivity of arrows from our system, if such \(\nu\) exists, it need to be unique.
We have
\[
\displaystyle h_n \circ f_{1}^{n}
\displaystyle = h_{1}
\displaystyle = \nu \circ f_{1}
\displaystyle = \left(\nu \circ f_{n}\right) \circ f_{1}^{n},
\]
so, by the surjectivity of \(f_1^n\), the equality \(\nu \circ f_n = h_n\) holds for all \(n \in \mathbb{N}\).
Therefore, it is sufficient for this to~be true, to find \(\nu\) such that \(\nu \circ f_1 = h_1\).
Define then \(\nu\) to satisfy this formula.
We will prove that \(\nu\) is~a~well defined arrow in \(\mathsf{cmn}\).
Fix \(x\), \(y \in X_1\).
Then
\[
\displaystyle \zeta\left(h_{1}(x), \ h_{1}(y)\right)
\displaystyle = \zeta\left(\left(h_{\eta_{n}} \circ f_{1}^{\eta_{n}}\right)(x), \ \left(h_{\eta_{n}} \circ f_{1}^{\eta_{n}}\right)(y)\right)
\displaystyle \leqslant \varrho\left(f_{1}^{\eta_{n}}(x), \ f_{1}^{\eta_{n}}(y)\right).
\]
Basing on the inequalities proven at the beggining, we can pass with \(n\) to the infinity and get the inequality \linebreak \(\zeta(h_{1}(x)\), \(h_{1}(y)) \leqslant \varrho(f_1(x)\), \(f_1(y))\), from which we infer that \(\nu\) is well defined and non--expansive.
This~completes the prove on existence of inductive limit.

To prove that the object of our limit is unique up to isometry, it is enough to notice that, in fact, we have already proven that every subsequence of \(\{X_n\}_{n \in \mathbb{N}}\) contains a subsequence converging in \(\mathfrak{GH}\) to~this~object.
Therefore, the whole sequence \(\{X_n\}_{n \in \mathbb{N}}\) converges to the object of the direct limit which, then, is~independent of the choice of surjective arrows of our direct system.
This proves (\ref{main_N f1}) of the ``furthermore'' part of~this~theorem.
Let the pair \((Z\), \(\zeta)\), \(\{h_n \colon X_n \longrightarrow Z\}_{n \in \mathbb{N}}\) be another inductive limit of our system.
Then its~arrows have~to be surjective, as all of the elements of \(\{f_n\}_{n \in \mathbb{N}}\) are surjective, and the only arrow \(\nu \colon X \longrightarrow Z\) satisfying for every natural number \(n\) equation \(\nu \circ f_n = h_n\) is an isomorphism.
\end{proof}

We will now prove this theorem in the general case.
To accomplish this, we will use the Arzel\'{a}--Ascoli theorem to find a single arrow for our limit, instead of constructing it.

\begin{thm} \label{main}
Let \((\{(X_\sigma\), \(d_\sigma)\}_{\sigma \in \Sigma}\), \(\{\alpha_{\sigma}^{\tau} \colon X_\sigma \longrightarrow X_\tau\}_{\sigma \leqslant \tau})\) be a direct system in \(\mathsf{cmn}\), where \(\alpha_\sigma^\tau\) are surjective for all \(\sigma \leqslant \tau\).
Then this system has inductive limit.\\
Furthermore,
\begin{enumerate}
\item its object does not depend, up~to~isometry, on the choice of surjective arrows  for the system, and \label{main f1}
\item its arrows are surjective. \label{main f2}
\end{enumerate}
\end{thm}

\begin{proof}
Similarly as in case of Theorem \ref{main_N}, our first and most important step of the proof will be to produce a single surjective arrow for~our~inductive limit.
Begin then by fixing \(\sigma_0 \in \Sigma\) and, similarly as in~the~proof for~\(\Sigma = \mathbb{N}\), let us embed \(\{X_\sigma\}_{\sigma \geqslant \sigma_0}\) isometrically in some compact metric space \((Y\), \(\varrho)\), and denote, after~passing to a subnet, \(X\) to be the Hausdorff limit of \(\{X_\sigma\}_{\sigma \geqslant \sigma_0}\) in \(Y\).
Of course, our aim now will be to~prove that \(X\) is the object of the direct limit of our system.

By the above, we can consider the family \(\{\alpha_{\sigma_0}^\sigma \colon X_{\sigma_0} \longrightarrow Y\}_{\sigma \geqslant \sigma_0}\) of arrows.
This family clearly fulfills the assumptions of the Arzel\'{a}--Ascoli theorem.
Therefore, we know that it has a subnet uniformly convergent to~some non--expansive mapping \(\alpha_{\sigma_0} \colon X_{\sigma_0} \longrightarrow Y\).
Once again,  for convenience, we can assume that~\(\{\alpha_{\sigma_0}^{\sigma}\}_{\sigma \geqslant \sigma_0}\) converges uniformly to \(\alpha_{\sigma_0}\).
Combining this fact with the convergence of \(\{X_{\sigma}\}_{\sigma \geqslant \sigma_0}\) to \(X\), it~is easy to check that \(\OPN{im}(\alpha_{\sigma_0}) = X\).
We have a single surjective arrow for our direct limit.
It remains, basing on \(\alpha_{\sigma_0}\), to define the rest of them.

Let us define \(\alpha_\sigma \colon X_\sigma \longrightarrow X\) for \(\sigma \in \Sigma\).
If \(\sigma \geqslant \sigma_0\), this mapping will be given by the rule: \linebreak \(\alpha_{\sigma}(x) \in \alpha_{\sigma_0}((\alpha_{\sigma_0}^\sigma)^{-1}(\{x\}))\) for~any \(x \in X_\sigma\), whereas if \(\sigma \not\geqslant \sigma_0\), we define \(\alpha_{\sigma}\) as \(\alpha_{\sigma'} \circ \alpha_{\sigma}^{\sigma'}\), where \(\sigma' \in \Sigma\) is~chosen arbitrarily so that \(\sigma' \geqslant \sigma_0\) and \(\sigma' \geqslant \sigma\).
The task is now to~prove that these mappings are well~defined arrows.
Let us first assure ourselves that this is the case for arbitrary \(\sigma \geqslant \sigma_0\).
For this purpose, we should remind ourselves of a final subnet of \(\Sigma\) from our previous considerations.
Let us denote the set of indices of~this~subnet by \(\{\eta_\lambda\}_{\lambda \in \Lambda}\).
Fix \(x_1\), \(x_2 \in X_\sigma\).
Choose \(a_j \in (\alpha_{\sigma_0}^\sigma)^{-1}(\{x_j\})\) and let \(y_j = \alpha_{\sigma_0}(a_j)\)  for~\(j = 1\),~\(2\).
Observe that \(\alpha_{\sigma_0}^{\eta_\lambda} = \alpha_{\sigma}^{\eta_\lambda} \circ \alpha_{\sigma_0}^{\sigma}\) for sufficiently large \(\lambda \in \Lambda\).
Because of this and the fact that~\(\{\alpha_{\sigma_0}^{\eta_\lambda}\}_{\lambda \in \Lambda}\) is~pointwise convergent to \(\alpha_{\sigma_0}\), for~\(j = 1\),~\(2\) we~have~that
\[
\displaystyle y_j
\displaystyle = \lim_{\lambda \, \in \, \Lambda} \alpha_{\sigma_0}^{\eta_\lambda}\left(a_j\right)
\displaystyle = \lim_{\lambda \, \in \, \Lambda} \alpha_{\sigma}^{\eta_\lambda}\left(\alpha_{\sigma_0}^{\sigma}\left(a_j\right)\right)
\displaystyle = \lim_{\lambda \, \in \, \Lambda} \alpha_{\sigma}^{\eta_\lambda}\left(x_j\right).
\]
It follows that the right hand side of the inequality
\(
\varrho(x_1\), \(x_2) \geqslant \varrho(\alpha_{\sigma}^{\eta_\lambda}(x_1)\), \(\alpha_{\sigma}^{\eta_\lambda}(x_2))
\)
converges to~\(\varrho(y_1\),~\(y_2)\).
Thus \(y_1 = y_2\) if only \(x_1 = x_2\), so \(\alpha_\sigma\) is well defined and non--expansive.
Of course, it is also surjective.
We~now turn to the case when \(\sigma \not\geqslant \sigma_0\).
Let us first check that \(\alpha_{\sigma''} \circ \alpha_{\sigma'}^{\sigma''} = \alpha_{\sigma'}$ for $\sigma'' \geqslant \sigma' \geqslant \sigma_0\).
To~see~this, it~is enough to notice that \[
\displaystyle \left(\alpha_{\sigma_0}^{\sigma'}\right)^{-1}(\{x\}) 
\displaystyle \subset \left(\alpha_{\sigma_0}^{\sigma'}\right)^{-1}\left(\left(\alpha_{\sigma'}^{\sigma''}\right)^{-1}\left(\left\{\alpha_{\sigma'}^{\sigma''}(x)\right\}\right)\right)
\displaystyle = \left(\alpha_{\sigma_0}^{\sigma''}\right)^{-1}\left(\left\{\alpha_{\sigma'}^{\sigma''}(x)\right\}\right)
\]
for any \(x \in X_{\sigma'}\).
Consequently, since \(\Sigma\) is upward directed, \(\alpha_\sigma\) is well defined.
Of course, it is surjective and non--expansive.

We already have a family of well defined surjective arrows \(\{\alpha_{\sigma} \colon X_\sigma \longrightarrow X\}_{\sigma \in \Sigma}\).
What is left to show \linebreak is that this family, together with \(X\), is the inductive limit of our direct system.
Let us now check the formula \(\alpha_{\sigma'} \circ \alpha_{\sigma}^{\sigma'} = \alpha_{\sigma}\) for~\(\sigma' \geqslant \sigma\).
The case when \(\sigma' \geqslant \sigma_0\) is obvious.
If \(\sigma' \not\geqslant \sigma_0\), there is \(\sigma'' \in \Sigma\) such that \(\sigma'' \geqslant \sigma_0\)~and \(\sigma'' \geqslant \sigma'\).
Then
\[
\displaystyle \alpha_{\sigma}
\displaystyle = \alpha_{\sigma''} \circ \alpha_{\sigma}^{\sigma''}
\displaystyle = \alpha_{\sigma''} \circ \alpha_{\sigma'}^{\sigma''} \circ \alpha_{\sigma}^{\sigma'}
\displaystyle = \alpha_{\sigma'} \circ \alpha_{\sigma}^{\sigma'}.
\]

Fix a compact metric space \((Z\), \(\zeta)\) and a family of arrows \(\{\xi_\sigma \colon X_\sigma \longrightarrow Z\}_{\sigma \in \Sigma}\) such that \(\xi_{\sigma} = \xi_{\sigma'} \circ \alpha_{\sigma}^{\sigma'}\) for~\(\sigma' \geqslant \sigma\) in \(\Sigma\).
It remains to prove that there exists an arrow \(\nu \colon X \longrightarrow Z\) satisfying \(\nu \circ \alpha_{\sigma} = \xi_{\sigma}\) for \(\sigma \in \Sigma\).
Suppose for~a~while that we have already found such \(\nu\) that satisfies \(\nu \circ \alpha_{\sigma_0} = \xi_{\sigma_0}\) (note that this forces \(\nu\) to be unique).
In such a situation, if \(\sigma \geqslant \sigma_0\), we have 
\[
\displaystyle \xi_\sigma \circ \alpha_{\sigma_0}^{\sigma}
\displaystyle = \xi_{\sigma_0}
\displaystyle = \nu \circ \alpha_{\sigma_0}
\displaystyle = \left(\nu \circ \alpha_{\sigma}\right) \circ \alpha_{\sigma_0}^{\sigma}.
\]
By~the surjectivity of \(\alpha_{\sigma_0}^\sigma\), we have the desired equality.
On the other hand, if \(\sigma \not\geqslant \sigma_0\), there exists \(\sigma' \in \Sigma\) such that \(\sigma' \geqslant \sigma_0\) and \(\sigma' \geqslant \sigma\) in \(\Sigma\).
Then we know that 
\[
\xi_{\sigma}
\displaystyle = \xi_{\sigma'} \circ \alpha_{\sigma}^{\sigma'}
\displaystyle = \nu \circ \alpha_{\sigma'} \circ \alpha_\sigma^{\sigma'}
\displaystyle = \nu \circ \alpha_{\sigma}.
\]
Thus, if the equation \(\nu \circ \alpha_{\sigma_0} = \xi_{\sigma_0}\) correctly defines a non--expansive map \(\nu\), the assertion follows. \linebreak
For \(x\), \(y \in X_{\sigma_0}\) and \(\lambda \in \Lambda\),
\[
\displaystyle \zeta\left(\xi_{\sigma_0}(x), \ \xi_{\sigma_0}(y)\right)
\displaystyle = \zeta\left(\left(\xi_{\eta_{\lambda}} \circ \alpha_{\sigma_0}^{\eta_{\lambda}}\right)(x), \ \left(\xi_{\eta_{\lambda}} \circ \alpha_{\sigma_0}^{\eta_{\lambda}}\right)(y)\right)
\displaystyle \leqslant \varrho\left(\alpha_{\sigma_0}^{\eta_{\lambda}}(x), \ \alpha_{\sigma_0}^{\eta_{\lambda}}(y)\right).
\]
Passing to the limit we obtain \(\zeta(\xi_{\sigma_0}(x)\),~\(\xi_{\sigma_0}(y)) \leqslant \varrho(\alpha_{\sigma_0}(x)\), \(\alpha_{\sigma_0}(y))\), which implies that \(\nu\) is well defined~and non--expansive.
This ends the proof of the first part of the theorem.

We have proven that any subnet of \(\{X_\sigma\}_{\sigma \in \Sigma}\) contains a subnet convergent to the object of the direct limit.
It means that \(\{X_\sigma\}_{\sigma \in \Sigma}\) is convergent, in the Gromov--Hausdorff space, to this object.
This shows point (\ref{main f1}) of the~second part of the theorem.
Let the pair \(((Z\), \(\zeta)\), \(\{\xi_\sigma \colon X_\sigma \longrightarrow Z\}_{\sigma \in \Sigma})\) be another inductive limit of our system.
Then its arrows are surjective, because all of the elements of \(\{\alpha_\sigma\}_{\sigma \in \Sigma}\) are surjective, and the only arrow \(\nu \colon X \longrightarrow Z\) satisfying for every \(\sigma \in \Sigma\) equation \(\nu \circ \alpha_\sigma = \xi_\sigma\) is an isomorphism.
\end{proof}

\begin{cor} \label{Convergence_of_a_net}
If the pair \((\{(X_\sigma\), \(d_\sigma)\}_{\sigma \in \Sigma}\), \(\{\alpha_{\sigma}^{\tau} \colon X_\sigma \longrightarrow X_\tau\}_{\sigma \leqslant \tau})\) fulfills the assumptions of the previous theorem, then the net \(\{X_\sigma\}_{\sigma \in \Sigma}\) is convergent in the Gromov--Hausdorff space.
\end{cor}

It is worth to consider the partial order \(\preccurlyeq\) on \(\mathfrak{GH}\) (as in \cite{Ultrametrics}) defined by
\[
(X, \, d) \preccurlyeq (Y, \, \varrho) \overset{\OPN{def}}{\Longleftrightarrow} \varphi(Y) = X \textnormal{ for some non--expansive mapping } \varphi \colon Y \longrightarrow X, 
\]
as it, combined with the above Corollary \ref{Convergence_of_a_net}, leads to an interesting conclusion being an analogue of~the~classical property of sequences in \(\mathbb{R}\).
Namely, the aforementioned corollary allows us to infer that every decreasing (with respect to \(\preccurlyeq\)) sequence of compact metric spaces is~convergent in \(\mathfrak{GH}\).

\begin{cor} \label{convergence-decreasing}
Let \(\{(X_n\), \(d_n)\}_{n \in \mathbb{N}}\) be a decreasing (with respect to \(\preccurlyeq\)) sequence in \(\mathfrak{GH}\).
Then it has a limit in \(\mathfrak{GH}\).
\end{cor}

To get a complete analogy to the situation in \(\mathbb{R}\), that is, the theorem on the convergence of monotonic and~bounded sequences, one need to confirm that every increasing and bounded from above (with~respect~to~\(\preccurlyeq\)) sequence of~compact metric spaces is convergent in \(\mathfrak{GH}\).
(Note that every sequence in \(\mathfrak{GH}\) is~bounded from below.)
We will take the path to this minor goal similarly as we did the previous time.
We~will begin by~proving a much more general theorem on the inverse limit of bounded (from above) inverse system in~\(\mathsf{cmn}\), and the aforementioned property will only be the inference of this assertion.

\subsection{Inverse limit}
As before, begin by recalling the definition of an \emph{inverse limit} of an inverse system.

\begin{deff} \label{Inverse limit}
Let \((\{A^\sigma\}_{\sigma \in \Sigma}\), \(\{\alpha_{\tau}^{\sigma} \colon A^\tau \longrightarrow A^\sigma\}_{\sigma \leqslant \tau})\) be an inverse system in~a~category \(\textsf{A}\).
By its \emph{inverse limit} we mean an object \(A\) in \(\textsf{A}\) together with a family of~morphisms \(\{\alpha^\sigma \colon A \longrightarrow~A^\sigma\}\) in \(\textsf{A}\) such~that
\begin{enumerate}
\item \(\alpha_{\tau}^{\sigma} \circ \alpha^\tau  = \alpha^\sigma\) for \(\sigma \leqslant \tau\) in \(\Sigma\), \label{Inverse limit 1}
\end{enumerate}
and the following statement is fulfilled:
\begin{enumerate}
\setItemnumber{2}
\item \label{Inverse limit 2} if \(A'\) is an object in \(\textsf{A}\) and a family \(\{\xi^{\sigma} \colon A' \longrightarrow A^\sigma\}_{\sigma \in \Sigma}\) of morphisms in \(\textsf{A}\) are such~that
\(\alpha_{\tau}^{\sigma} \circ \xi^\tau = \xi^\sigma\) for \(\sigma \leqslant \tau\) in \(\Sigma\),
then there exists a unique morphism \(\nu \colon A' \longrightarrow A\) such that \(\alpha^\sigma \circ \nu = \xi^\sigma\) for any \nolinebreak \(\sigma \in \Sigma\).
\end{enumerate}
\end{deff}

Also this time, if this limit exists (which is not always the case), then its object is unique up~to~isomorphism.

\begin{lem} \label{inverse limit - main}
Let \((\{(X^\sigma\), \(d^\sigma)\}_{\sigma \in \Sigma}\), \(\{\alpha^{\sigma}_{\tau} \colon X^\tau\longrightarrow X^\sigma\}_{\sigma \leqslant \tau})\) be an inverse system in \(\mathsf{cmn}\), where \(\alpha^\sigma_\tau\) are~surjective for all \(\sigma \leqslant \tau\).
Then this system has inverse limit if and only if there exists an object in \(\mathsf{cmn}\) and surjective arrows from this object onto \(X^\sigma\) for \(\sigma \in \Sigma\) in \(\mathsf{cmn}\).\\
Furthermore, if the limit exists,
\begin{enumerate}
\item its object does~not depend, up~to~isometry, on the choice of surjective arrows for the system, and \label{inverse limit - main f1}
\item its arrows are surjective. \label{inverse limit - f2}
\end{enumerate}
\end{lem}

\begin{proof}
We first prove that condition on the boundedness (with respect to \(\preccurlyeq\)) of~the~set of objects of~the~inverse system is sufficient for the existence of an inverse limit.
At the beginning, similarly as in the proof of~Theorem~\ref{main}, thanks to the Gromov's theorem, basing~on~the~fact that the family \(\{X^\sigma\}_{\sigma \in \Sigma}\) is bounded from above with respect to \(\preccurlyeq\), we embed isometrically its elements in some compact metric space \((Y\), \(\varrho)\).

We will now produce an object and arrows 
for our inverse limit.
To do so, we will use Tikhonov's theorem.
Consider
\( \gamma_\tau =~(X^\tau\),~\(\{\Gamma(\beta_{\tau}^{\sigma})\}_{\sigma \in \Sigma}) \in \mathfrak{K}(Y) \times \prod_{\sigma \in \Sigma} \mathfrak{K}(Y \times Y) \) for \(\tau \in \Sigma\), where
\(\mathfrak{K}(Y)\) and \(\mathfrak{K}(Y \times Y)\) are~the~hyperspaces of, respectively, \(Y\) and \(Y \times Y\) (considered with the sum metric \(\varrho_1\)), \(\Gamma(\beta_{\tau}^{\sigma})\) are the graphs of~mappings \(\beta_{\tau}^{\sigma}\) for \(\sigma\), \(\tau \in \Sigma\), while~\(\beta_{\tau}^{\sigma} =~\alpha_{\tau}^{\sigma}\) if \(\sigma \leqslant \tau\) and \(\beta_{\tau}^{\sigma} = \OPN{id}_{X^\tau}\) otherwise.
By the Tikhonov's theorem, we find a convergent subnet \( \{\gamma_{\eta_{\lambda}}\}_{\lambda \in \Lambda} \) of~\( \{\gamma_\tau\}_{\tau \in \Sigma}\) in \(\mathfrak{K}(Y) \times \prod_{\sigma \in \Sigma} \mathfrak{K}(Y \times Y)\).
We see thus, that~for~every \(\sigma \in \Sigma\) the net \( \{\Gamma(\beta_{\eta_{\lambda}}^{\sigma})\}_{\lambda \in \Lambda}\) converges to some \( Y^\sigma \subset \mathfrak{K}(Y \times Y) \).
Note that if \( \mathfrak{a}\), \(\mathfrak{b} \in Y_\sigma\), then, by~the~Hausdorff convergence, there exist \(\{\mathfrak{a}_{\eta_\lambda}\}_{\lambda \in \Lambda}\), \(\{\mathfrak{b}_{\eta_\lambda}\}_{\lambda \in \Lambda} \in \prod_{\lambda \in \Lambda} \Gamma(\beta_{\eta_\lambda}^{\sigma})\) such that \(\mathfrak{a}_{\eta_\lambda} \longrightarrow \mathfrak{a}\) and \(\mathfrak{b}_{\eta_\lambda} \longrightarrow \mathfrak{b}\). \linebreak
Let \(p_j \colon Y \times Y \longrightarrow Y\) for~\(j = 1\), \(2\) be the projection onto the \(j\)'th coordinate.
Because
\[
\displaystyle \varrho(p_2(\mathfrak{a}_{\eta_\lambda}), \, p_2(\mathfrak{b}_{\eta_\lambda}))
\leqslant
\displaystyle \varrho(p_1(\mathfrak{a}_{\eta_\lambda}), \, p_1(\mathfrak{b}_{\eta_\lambda}))
\]
for every \(\lambda \in \Lambda\), by taking limits on both sides of the above inequality we see that \(Y^\sigma\)~is a~graph of~a~non--expansive mapping.
Denote that mapping by \(\alpha^\sigma\).
Denote also \(X\) to be the Hausdorff limit of~\(\{X^{\eta_\lambda}\}_{\lambda \in \Lambda}\)~and note, as mappings \(p_1\), \(p_2\) are continuous, that \(X^{\eta_\lambda} = p_1(\Gamma(\beta_{\eta_\lambda}^{\sigma})) \longrightarrow p_1(Y^\sigma) = X\)~and \(X^{\sigma} = p_2(\Gamma(\beta_{\eta_\lambda}^{\sigma})) \longrightarrow~p_2(Y^\sigma) = X^\sigma\).
So, we have already found an object \(X \in \mathsf{cmn}\) together with~the~set of non--expansive (and surjective) arrows \(\{\alpha^\sigma \colon X \longrightarrow X^\sigma\}_{\sigma \in \Sigma}\).
We thus have ingredients that may be constituent parts of the inverse limit we are looking for. 

Now, we prove point (\ref{Inverse limit 1}) of Definition \ref{Inverse limit}.
We will infer this from another fact.
Fix \(\sigma \in \Sigma\).
We will show that
\begin{equation} \label{eq 3}
\displaystyle X
=
\displaystyle \bigcup_{x^\sigma \, \in \, X^{\sigma}} \bigcup\left\{K \colon \ K \textnormal{ is a cluster point of } \left\{\left(\alpha_{\eta_{\lambda}}^{\sigma}\right)^{-1}(\{x^\sigma\})\right\}_{\lambda \geqslant \lambda_0} \textnormal{ in } \mathfrak{K}(Y)\right\},
\end{equation}
where \(\lambda_0\) is such that \(\eta_\lambda \geqslant \sigma\) for every \(\lambda \geqslant \lambda_0\).
(Of course, by the Hausdorff convergence of \(\{X^{\eta_\lambda}\}_{\lambda \geqslant \lambda_0}\), the right hand side of \eqref{eq 3} is contained in \(X\).)
Fix \(x \in X\).
By the convergence of \(\{X^{\eta_\lambda}\}_{\lambda \geqslant \lambda_0}\), we~find a~net \(\{x_{\eta_\lambda}^{\eta_\lambda}\}_{\lambda \geqslant \lambda_0} \in \prod_{\lambda \geqslant \lambda_0} X^{\eta_\lambda}\) that converges to \(x\) with respect to \(\varrho\).
Let \(x_{\eta_\lambda}^{\eta_\kappa} = \alpha_{\eta_\lambda}^{\eta_\kappa}(x_{\eta_\lambda}^{\eta_\lambda}) \in X^{\eta_\kappa}\) for all \(\kappa\),~\(\lambda \in \Lambda\) such~that \(\sigma \leqslant \eta_\kappa \leqslant \eta_\lambda\).
Thanks to Tikhonov's theorem, without~loss~of~generality one can assume that~for~every \(\kappa \geqslant \lambda_0\) the net \(\{x_{\eta_\iota}^{\eta_\kappa}\}_{\iota : \, \eta_\iota \geqslant \eta_\kappa}\) is convergent to~some \(x^{\eta_\kappa} \in X^{\eta_\kappa}\).
Once again, without loss of~generality, we may and do assume that \(\{x^{\eta_\kappa}\}_{\kappa \geqslant \lambda_0}\) converges to~\(z \in X\).
Note that if \(\kappa\), \(\lambda \in \Lambda\) are such that \(\sigma \leqslant \eta_\kappa \leqslant \eta_\lambda\), then by the continuity of \(\alpha_{\eta_\lambda}^{\eta_\kappa}\) we have \(\alpha_{\eta_\lambda}^{\eta_\kappa}(x^{\eta_\lambda}) = x^{\eta_\kappa}\).
In case \(\eta_{\lambda_0} \neq \sigma\), we~set~\(x^\sigma = \alpha_{\eta_{\lambda_0}}^{\sigma}(x^{\eta_\lambda})\), and, in fact, by the above we get that \(\alpha_{\eta_\lambda}^{\sigma}(x^{\eta_\lambda}) = x^{\sigma}\) for every \(\lambda \geqslant \lambda_0\).
Therefore, by the facts that \(\{\Gamma(\alpha_{\eta_\lambda}^\sigma)\}_{\lambda \geqslant \lambda_0}\) converges to \(\Gamma(\alpha^\sigma)\) in \(\mathfrak{K}(Y \times Y)\) and \(\{x^{\eta_\lambda}\}_{\lambda \geqslant \lambda_0}\) converges to \(z\) in \(Y\), we infer that \(\alpha^\sigma(z) = x^\sigma\).
Simultaneously, we have shown that \(\{x^{\eta_\kappa}\}_{\kappa \geqslant \lambda_0} \in~\prod_{\kappa \geqslant \lambda_0}(\alpha_{\eta_{\kappa}}^{\sigma})^{-1}(\{x^\sigma\})\).
Thanks to these last two facts, to~prove~\eqref{eq 3}, it remains to show that \(x = z\).
To see this, consider the net of arrows \(\{\alpha^{\eta_\lambda}\}_{\lambda \geqslant \lambda_0}\).
By~the~Arzel\'{a}--Ascoli theorem, without loss of generality we assume that it is convergent to some non--expansive surjective mapping \(\alpha \colon X \longrightarrow X\).
Hence, it is also isometric.
Therefore, we have
\[
\displaystyle 0 = \varrho(\alpha^{\eta_\kappa}(x), \, \alpha^{\eta_\kappa}(z))
\displaystyle \overset{\kappa \geqslant \lambda_0}{\longrightarrow} \varrho(\alpha(x), \, \alpha(z))
\displaystyle = \varrho(x, \, z),
\]
so \eqref{eq 3} holds.
(It is worth noting that in particular we got an~information that \(\alpha = \OPN{id}_X\), for \(x = z\), and~that~\(\alpha^{\eta_\lambda}(z) = x^{\eta_\lambda}\) for every \(\lambda \geqslant \lambda_0\), what can be seen in a similar way as the fact that \(\alpha^\sigma(z) = x^\sigma\).)

Of course, thanks to \eqref{eq 3} we see that \(\alpha^\sigma\) for every \(\sigma \in \Sigma\) is in fact defined by the rule
\begin{equation} \label{eq 4}
\displaystyle \left(\alpha^{\sigma}\right)^{-1}(\{x\})
=
\displaystyle \bigcup\left\{K \colon \ K \textnormal{ is a cluster point of } \left\{\left(\alpha_{\sigma}^{\eta_\lambda}\right)^{-1}(\{x\})\right\}_{\lambda \geqslant \lambda_0}\right\}
\end{equation}
for \(x \in X^{\sigma}\).
Fix \(\sigma \leqslant \tau\) in \(\Sigma\) and \(x \in X\).
To see that \(\alpha_\tau^\sigma (\alpha^\tau (x)) = \alpha^\sigma (x)\), it is enough to use the above rule together with the fact that
\[
\left(\alpha^{\tau}_{\eta_\lambda}\right)^{-1}(\left\{\alpha^\tau (x)\right\})
\subset 
\left(\alpha^{\tau}_{\eta_\lambda}\right)^{-1} \left((\alpha_{\tau}^{\sigma}\right)^{-1} (\left\{\alpha^\sigma (x)\right\}))
\]
for every \(\lambda \geqslant \lambda_0\), where \(\lambda_0 \in \Lambda\) is such that \(\eta_\lambda \geqslant \tau\) if only \(\lambda \geqslant \lambda_0\).

Assume now that there is a compact metric space \((Z\), \(\zeta)\) and a family of arrows \(\{\xi^\sigma \colon Z\longrightarrow X^\sigma \}_{\sigma \in \Sigma}\) such~that \(\alpha_{\sigma'}^{\sigma} \circ \xi^{\sigma'} = \xi^{\sigma}\) for~\(\sigma' \geqslant \sigma\) in \(\Sigma\).
To finish the main part of~the~proof, it is enough to show that~there exists an arrow \(\nu \colon Z \longrightarrow X\) satisfying \(\alpha^\sigma \circ \nu = \xi^\sigma\) for \(\sigma \in \Sigma\).
Similarly as in the previous part of~the~proof, using Arzel\'{a}--Ascoli theorem we find subnet \(\{\alpha^{\eta_\lambda}\}_{\lambda \in \Lambda}\) of \(\{\alpha^{\sigma}\}_{\sigma \in \Sigma}\) that converges uniformly to \(\OPN{id}_X\).
Assume for a while that there exists some \(\nu \colon Z \longrightarrow X\) such that \(\alpha^\sigma \circ \nu = \xi^\sigma\) for \(\sigma \in \Sigma\).
Then it has to be unique.
Indeed,
it is because if \(x\), \(y \in X\) are such that \(\alpha^\sigma(x) = \alpha^\sigma(y)\) for every \(\sigma \in \Sigma\), then, in particular,
\[
\displaystyle 0 = \varrho(\alpha^{\eta_\lambda}(x), \, \alpha^{\eta_\lambda}(y))
\displaystyle \overset{\lambda \in \Lambda}{\longrightarrow}
\displaystyle \varrho(x, \, y),
\]
so \(x = y\).
Without loss of generality we may and do assume that \(\{\xi^{\eta_\lambda}\}_{\lambda \in \Lambda}\) is uniformly convergent.
Of~course, we set \(\nu = \lim_{\lambda \in \Lambda} \xi^{\eta_\lambda}\).
Note that, thanks to the rule from \eqref{eq 4} and the fact that \(\alpha_{\sigma'}^{\sigma} \circ \xi^{\sigma'} = \xi^{\sigma}\) for~\(\sigma' \geqslant \sigma\) in~\(\Sigma\), we~immediately obtain that \(\alpha^\sigma \circ \nu = \xi^\sigma\) for \(\sigma \in \Sigma\).
If \(\sigma \notin \{\eta_\lambda\}_{\lambda \in \Lambda}\), then we find \(\lambda \in \Lambda\) such that \(\lambda \geqslant \sigma\), and then 
\[
\displaystyle \alpha^\sigma \circ \nu
=
\displaystyle \alpha_{\eta_\lambda}^\sigma \circ (\alpha^{\eta_\lambda} \circ \nu)
=
\displaystyle \alpha_{\eta_\lambda}^\sigma \circ \xi^{\eta_\lambda}
=
\displaystyle \xi^\sigma.
\]

Now we turn to the proof of point (\ref{inverse limit - f2}) of the ``furthermore'' part.
(Note that this will entail the second implication of the main part of the theorem.)
Assume that there is an object \(X\) and a set of arrows \(\{\alpha^\sigma \colon X \longrightarrow~X^\sigma\}_{\sigma \in \Sigma}\) in \(\mathsf{cmn}\) such that the condition (\ref{Inverse limit 1}) from Definition~\ref{Inverse limit} holds, and that there is \(\sigma_0 \in \Sigma\) such that \(\alpha^{\sigma_0}\) is not surjective.
We will prove that~in~this~case the pair \(((X\), \(d)\), \(\{\alpha^\sigma \colon X \longrightarrow X^\sigma\}_{\sigma \in \Sigma})\) is not an inverse limit of~our~system.
Note that, thanks~to our~assumption, there is \(x^{\sigma_0} \in X^{\sigma_0}\) such that \(x^{\sigma_0} \notin \alpha^{\sigma_0}(X)\).
Consider for all \(\sigma \leqslant \tau\) in~\(\Sigma\) the set
\[
\displaystyle T_{\tau}^{\sigma} \defeq \left\{ \left\{ x_\sigma \right\}_{\sigma \in \Sigma} \in \prod_{\sigma \, \in \, \Sigma} X^\sigma \colon \ x_{\sigma_0} = x^{\sigma_0} \textnormal{ and } \alpha_{\tau}^{\sigma}(x_\tau) = x_\sigma \right\}.
\]
Clearly, every such set is a non--empty compact subset of \(\prod_{\sigma \in \Sigma} X^\sigma\) (considered with the product topology).
Define also for \(\tau \in \Sigma\)
\[
\displaystyle T_{\tau} \defeq \left\{ \left\{ x_\sigma \right\}_{\sigma \in \Sigma} \in \prod_{\sigma \, \in \, \Sigma} X^\sigma \colon \ x_{\sigma_0} = x^{\sigma_0} \textnormal{ and } \alpha_{\tau}^{\sigma}(x_\tau) = x_\sigma \textnormal{ for every } \sigma \leqslant \tau \right\}.
\]
Note that, if \(\sigma_j \leqslant \tau\) for \(\rowto{j}{n}\), where \(n \in \mathbb{N}\), then obviously the intersection \(\bigcap_{j = 1}^n T_{\tau}^{\sigma_j}\) is non--empty, and thus the family \(\{T_\tau^\sigma\}_{\sigma \leqslant \tau}\) is centred.
We infer from this fact that \(T_\tau = \bigcap_{\sigma \leqslant \tau} T_{\tau}^{\sigma_j}\) is non--empty compact subset of \(\prod_{\sigma \in \Sigma} X^\sigma\).
Once again, if \(\row{\tau}{n} \in \Sigma\) for some \(n \in \mathbb{N}\), then we find \(\tau \geqslant \tau_j\) for \(\rowto{j}{n}\).
In~this situation we have \(T_\tau \subset \bigcap_{j = 1}^{n} T_{\tau_j}\), so the family \(\{T_\tau\}_{\tau \in \Sigma}\) is centred and the set
\[
T \defeq \left\{ \left\{ x_\sigma \right\}_{\sigma \in \Sigma} \in \prod_{\sigma \, \in \, \Sigma} X^\sigma \colon \ x_{\sigma_0} = x^{\sigma_0} \textnormal{ and } \alpha_{\tau}^{\sigma}(x_\tau) = x_\sigma \textnormal{ for all } \sigma \leqslant \tau \right\},
\]
as an intersection of the aforementioned family, is non--empty.
We can thus fix an element \(\{x^\sigma\}_{\sigma \in \Sigma} \in T\).
Choose arbitrary \(X' \in \mathfrak{GH}\) and for every \(\sigma \in \Sigma\) define \(\xi^\sigma \colon X' \longrightarrow X^\sigma\) in the way that \(\xi^\sigma(x') = x^\sigma\) for~every \(x' \in X'\).
By the choice of \(\{x^\sigma\}_{\sigma \in \Sigma}\), this family clearly fulfills assumptions of (\ref{Inverse limit 2}) from~Definition~\ref{Inverse limit}, but~there is no arrow \(\nu \colon X' \longrightarrow X\) such that \(\alpha^{\sigma_0} \circ \nu = \xi^{\sigma_0}\).

To prove point (\ref{inverse limit - main f1}) of the ``furthermore'' part of the lemma, similarly as in the case of~the~proof for~direct system, it~is~sufficient to note that we have already proven that any subnet of \(\{X^\sigma\}_{\sigma \in \Sigma}\) contains a subnet convergent to~the object of the inverse limit, and thus \(\{X^\sigma\}_{\sigma \in \Sigma}\) has to be convergent in the Gromov--Hausdorff space to this object.
What is more, if the pair \(((Z\), \(\zeta)\), \(\{\xi^\sigma \colon Z\longrightarrow X^\sigma \}_{\sigma \in \Sigma})\) is another inverse limit of our system, then its arrows have to be surjective, because all of the elements of \(\{\alpha^\sigma\}_{\sigma \in \Sigma}\) are surjective, and there is an arrow \(\nu \colon X \longrightarrow Z\) satisfying for every \(\sigma \in \Sigma\) equation \(\alpha^{\sigma} \circ \nu = \xi^{\sigma}\).
\end{proof}

We will now give a simple example of an inverse system, which has no inverse limit.

\begin{ex}
Let \(X^n\) for \(n \in \mathbb{N}\) be the set \(\{\rowft{n}\}\) with discrete metric, and let \(\alpha_n^m \colon X^n \longrightarrow X^m\) for~natural numbers \(m \leqslant n\) be given by the formula
\[
\displaystyle \alpha_{n}^m(k)
\defeq
\left\{
\begin{array}{ll}
\displaystyle k, & \textnormal{ if } k \leqslant m,\\
\displaystyle 0, & \textnormal{ if } k > m.
\end{array}\right.
\]
Then, according to the above theorem, the system \((\{X^n\}_{n \in \mathbb{N}}\), \(\{\alpha_n^m\}_{m \leqslant n})\) does not have an inverse limit.
\end{ex}

If we consider in addition to \(\preccurlyeq\) the other classical partial orders on \(\mathfrak{GH}\) from \cite{Ultrametrics}, we get a very nice characterisation of the existence of an inverse limit of an inverse system.
These orders are defined as follows:
\[
\begin{array}{lll}
(X, \, d) \preccurlyeq_s (Y, \, \varrho) & \overset{\OPN{def}}{\Longleftrightarrow} & X \preccurlyeq K \textnormal{ for some closed subset } K \textnormal{ of } Y,\\
(X, \, d) \preccurlyeq_i (Y, \, \varrho) & \overset{\OPN{def}}{\Longleftrightarrow} & X \textnormal{ is isometric to a subset of } Y.
\end{array}
\]

Combining previous result with Theorem \(5.3\) from \cite{Ultrametrics} we obtain the following characterisation.

\begin{thm} \label{inverse limit - full version}
Let \((\{(X^\sigma\), \(d^\sigma)\}_{\sigma \in \Sigma}\), \(\{\alpha^{\sigma}_{\tau} \colon X^\tau\longrightarrow X^\sigma\}_{\sigma \leqslant \tau})\) be an inverse system in \(\mathsf{cmn}\), where \(\alpha^\sigma_\tau\) are~surjective for all \(\sigma \leqslant \tau\).
Then the following conditions are equivalent:
\begin{enumerate}
\item this system has inverse limit,
\item \(\{X^\sigma\}_{\sigma \in \Sigma}\) is upper bounded with respect to \(\preccurlyeq\),
\item \(\{X^\sigma\}_{\sigma \in \Sigma}\) is upper bounded with respect to \(\preccurlyeq_s\), \label{inverse limit - full version 3}
\item \(\{X^\sigma\}_{\sigma \in \Sigma}\) is upper bounded with respect to \(\preccurlyeq_i\),
\item there exists \(Y \in \mathfrak{GH}\) such that \(X^\sigma \preccurlyeq Y\), \(X^\sigma \preccurlyeq_s Y\) and \(X^\sigma \preccurlyeq_i Y\) for every \(\sigma \in \Sigma\),
\item \(\{X^\sigma\}_{\sigma \in \Sigma}\) is uniformly compact. \label{inverse limit - full version 6}
\end{enumerate}
Furthermore, if the limit exists, its object does~not depend, up~to~isometry, on the choice of surjective arrows for the system.
\end{thm}

For brevity, and because this partial order is the most natural one for our considerations, we will continue~to~use only \(\preccurlyeq\).
However, keep in mind that whenever we require from~the~family of compact metric spaces to~be bounded (or bounded from above) with respect to \(\preccurlyeq\), we can replace this~requirement with~any~of~(\ref{inverse limit - full version 3})--(\ref{inverse limit - full version 6}) from the previous theorem.

\begin{cor} \label{Convergence_of_a_net_2}
If the pair \((\{(X^\sigma\), \(d^\sigma)\}_{\sigma \in \Sigma}\), \(\{\alpha^{\sigma}_{\tau} \colon X^\tau\longrightarrow X^\sigma\}_{\sigma \leqslant \tau})\) is an inverse system as in the previous theorem and \(\{X^\sigma\}_{\sigma \in \Sigma}\) is bounded from above with respect to \(\preccurlyeq\), then the net \(\{X^\sigma\}_{\sigma \in \Sigma}\) is convergent in~the~Gromov--Hausdorff space.
\end{cor}

We infer from this that every increasing and bounded from above (with respect to \(\preccurlyeq\)) sequence of compact metric spaces is convergent in \(\mathfrak{GH}\).

\begin{cor} 
Let \(\{(X_n\), \(d_n)\}_{n \in \mathbb{N}}\) be an increasing and bounded from above (with respect to \(\preccurlyeq\)) sequence in~\(\mathfrak{GH}\).
Then it has a limit in \(\mathfrak{GH}\).
\end{cor}

This result, together with its previous version for decreasing sequences (see Corollary \ref{convergence-decreasing}), allows to state the analog for \(\mathfrak{GH}\) of the property of convergence of monotonic and bounded sequences in \(\mathbb{R}\). 

\begin{thm} \label{convergence of monotonic and bounded sequences}
Let \(\{(X_n\), \(d_n)\}_{n \in \mathbb{N}}\) be a monotonic and bounded (with respect to \(\preccurlyeq\)) sequence in \(\mathfrak{GH}\). \linebreak
Then it has a limit in \(\mathfrak{GH}\).
\end{thm}

\section{Compact metric groups with bi--invariant metrics} \label{Section 3}
\subsection{Inductive limit}
As announced, we will now prove the analog of Theorem \ref{main} in the version for metric groups with bi--invariant metrics.

\begin{thm} \label{main_groups}
Let \((\{(G_\sigma\), \(d_\sigma)\}_{\sigma \in \Sigma}\), \(\{\alpha_{\sigma}^{\tau} \colon G_\sigma \longrightarrow G_\tau\}_{\sigma \leqslant \tau})\) be a direct system in the category of compact metric groups with bi--invariant metrics with non--expansive homomorphisms as arrows, where \(\alpha_\sigma^\tau\) are~surjective for all \(\sigma \leqslant \tau\).
Then this system has inductive limit.\\
Furthermore,
\begin{enumerate}
\item its object does not depend, up~to~isometric isomorphism, on the choice of surjective arrows for~the~system, and \label{main_groups f1}
\item its arrows are surjective. \label{main_groups f2}
\end{enumerate}
\end{thm}

\begin{proof}
We will proceed similarly as in the proof of Theorem \ref{main}.
First, we fix arbitrarily \(\sigma_0 \in \Sigma\) and embed isometrically all the elements of the family \(\{G_{\sigma}\}_{\sigma \geqslant \sigma_0}\) in some compact metric space \((Y\), \(\varrho)\), and denote, after~passing to a subnet, \(G\) to be the Hausdorff limit of \(\{G_{\sigma}\}_{\sigma \geqslant \sigma_0}\) in \(Y\).
Let \(\alpha_{\sigma_0} \colon G_{\sigma_0} \longrightarrow G\) be the~uniform limit of the family \(\{\alpha_{\sigma_0}^{\sigma} \colon G_{\sigma_0} \longrightarrow Y\}_{\sigma \geqslant \sigma_0}\).
We already know that \(\alpha_{\sigma_0}\) is a non--expansive surjection.
It~should also be shown that it is also a homomorphism, and, before this, that one can define a group structure on~\(G\) in the way that the metric \(\varrho|_{G \times G}\) on \(G\) is bi--invariant.

We will now set a group structure on \(G\) by defining a binary operation on it.
Define \(\cdot \colon G \times G \longrightarrow G\) in~the~following way (for the convenience, we will also use this notation in the case of group operations for~the~remaining groups):
\[
\displaystyle g_1 \cdot g_2 \defeq \lim_{\sigma \geqslant \sigma_0} \alpha_{\sigma_0}^\sigma(g_1' \cdot g_2'),
\]
for \(g_1\), \(g_2 \in G\), where \(g_1'\), \(g_2' \in G_{\sigma_0}\) are such that \(g_j = \lim_{\sigma \geqslant \sigma_0} \alpha_{\sigma_0}^\sigma(g_j')\)  for \(j = 1\), \(2\).
This operation is~well~defined.
Indeed, if \(g_j \in G\) and \(g_j'\), \(g_j'' \in G_{\sigma_0}\) are such that 
\(
g_j = 
\lim_{\sigma \geqslant \sigma_0} \alpha_{\sigma_0}^\sigma(g_j')
=
\lim_{\sigma \geqslant \sigma_0} \alpha_{\sigma_0}^\sigma(g_j'')
\)
for~\(j = 1\), \(2\), then for \(\sigma \geqslant \sigma_0\)
\[
\begin{array}{lll}
\displaystyle d_\sigma\left(\alpha_{\sigma_0}^\sigma\left(g_1' \cdot g_2'\right), \, \alpha_{\sigma_0}^\sigma\left(g_1'' \cdot g_2''\right)\right)
& \leqslant &
\displaystyle d_{\sigma}\left(\alpha_{\sigma_0}^\sigma\left(g_1' \cdot g_2'\right), \, \alpha_{\sigma_0}^\sigma\left(g_1'' \cdot g_2'\right)\right)
+
\displaystyle d_{\sigma}\left(\alpha_{\sigma_0}^\sigma\left(g_1'' \cdot g_2'\right), \, \alpha_{\sigma_0}^\sigma\left(g_1'' \cdot g_2''\right)\right)\\
& = &
\displaystyle d_\sigma\left(\alpha_{\sigma_0}^\sigma\left(g_1'\right), \, \alpha_{\sigma_0}^\sigma\left(g_1''\right)\right)
+
\displaystyle d_\sigma\left(\alpha_{\sigma_0}^\sigma\left(g_2'\right), \, \alpha_{\sigma_0}^\sigma\left(g_2''\right)\right),
\end{array}
\]
while the right hand side of this inequality clearly converges to \(0\).
Verifying that this operation sets the~group structure on \(G\) is a straightforward computation and is left as an easy exercise.
What is more, by the definition of our operation, it is obvious that \(\alpha_{\sigma_0}\) is a homomorphism.

The fact that \(\varrho|_{G \times G}\) is bi--invariant is an immediate consequence of~the~fact that every metric  \(d_\sigma\) for~\(\sigma \geqslant \sigma_0\) is bi--invariant.
This means that \(G\) is a metric group with bi--invariant metric \(\varrho|_{G \times G}\).

The rest of the proof, except for the ``furthermore'' part, is carried out exactly like for analogous fragments of the proof of Theorem \ref{main}.
In turn, the proof of the aforementioned ``furthermore'' part is slightly more complitated than previously.
To prove it, we will use the fact that in the situation as in the above proof, if~for every \(\sigma \geqslant \sigma_0\) we have \(g_\sigma\), \(h_\sigma \in G_\sigma\) such that \(g_\sigma \overset{_{\sigma \geqslant \sigma_0}}{\longrightarrow} g\) and \(h_\sigma \overset{_{\sigma \geqslant \sigma_0}}{\longrightarrow} h\) (in this situtaion both \(g\) and~\(h\) are in \(G\)), then \(g_\sigma \cdot h_\sigma \overset{_{\sigma \geqslant \sigma_0}}{\longrightarrow} g \cdot h \in G\).
We will prove it similarly as the fact, that the operation \(\cdot\) on \(G\) is~well~defined.
Fix \(g'\) and \(h'\) in \(G_{\sigma_0}\) such that \(g = \lim_{\sigma \geqslant \sigma_0} \alpha_{\sigma_0}^\sigma(g')\) and \(h = \lim_{\sigma \geqslant \sigma_0} \alpha_{\sigma_0}^\sigma(h')\).
(We remember that then \(g \cdot h = \lim_{\sigma \geqslant \sigma_0} \alpha_{\sigma_0}^\sigma(g') \cdot \alpha_{\sigma_0}^\sigma(h')\).)
Using the fact that the metrics on our groups are bi--invariant, we can write
\[
\begin{array}{lll}
\displaystyle d_{\sigma}\left(g_\sigma \cdot h_\sigma, \, \alpha_{\sigma_0}^\sigma\left(g'\right) \cdot \alpha_{\sigma_0}^\sigma\left(h'\right)\right)
& \leqslant &
\displaystyle d_{\sigma}\left(g_\sigma \cdot h_\sigma, \, \alpha_{\sigma_0}^\sigma\left(g'\right) \cdot h_\sigma\right)
+
\displaystyle d_{\sigma}\left(\alpha_{\sigma_0}^\sigma\left(g'\right) \cdot h_\sigma, \, \alpha_{\sigma_0}^\sigma\left(g'\right) \cdot \alpha_{\sigma_0}^\sigma\left(h'\right)\right)\\
& = &
\displaystyle d_{\sigma}\left(g_\sigma, \, \alpha_{\sigma_0}^\sigma\left(g'\right)\right)
+
\displaystyle d_{\sigma}\left(h_\sigma, \, \alpha_{\sigma_0}^\sigma\left(h'\right)\right),
\end{array}
\]
what, after taking limit over \(\sigma \geqslant \sigma_0\), proves our property.

Let now \((H\), \(\{\beta_\sigma \colon G_\sigma \longrightarrow H\}_{\sigma \in \Sigma})\) be an inductive limit (that was found according to the above proof) of~a~direct system \((\{(G_\sigma, \, d_\sigma)\}_{\sigma \in \Sigma}\), \(\{\beta_{\sigma}^{\tau} \colon G_\sigma \longrightarrow G_\tau\}_{\sigma \leqslant \tau})\) satisfying assumptions of this theorem.
We~will~prove that there exists a non--expansive homomorphism from \(G\) onto \(H\).
Take a look at \(\{\beta_\sigma\}_{\sigma \geqslant \sigma_0}\) and \linebreak note that, similarly as in the proof of Lemma \ref{inverse limit - main}, there is a subnet \(\{\Gamma(\beta_{\eta_\lambda})\}_{\lambda \in \Lambda}\) of \(\{\Gamma(\beta_\sigma)\}_{\sigma \geqslant \sigma_0}\) \linebreak that~is~convergent in \(\mathfrak{K}(Y \times Y)\) to a graph of a non--expansive and surjective mapping \(\beta \colon G \longrightarrow H\).
To~see that \(\beta\) is a homomorphism, fix \((g\), \(\beta(g))\) and \((h\), \(\beta(h))\) in \(\Gamma(\beta)\), and for \(\lambda \in \Lambda\) find \(g_{\eta_\lambda}\), \(h_{\eta_\lambda}\) and \(g_{\eta_\lambda}'\), \(h_{\eta_\lambda}'\) in~\(G_{\eta_\lambda}\) such that \(g = \lim_{\lambda \in \Lambda} g_{\eta_\lambda}\), \(h = \lim_{\lambda \in \Lambda} h_{\eta_\lambda}\) and \(\beta(g) = \lim_{\lambda \in \Lambda} g_{\eta_\lambda}'\), \(\beta(h) = \lim_{\lambda \in \Lambda} h_{\eta_\lambda}'\).
We now use the proved in the above paragraph property of our inductive limits in case of \((G\), \(\{\alpha_\sigma \colon G_\sigma \longrightarrow G\}_{\sigma \in \Sigma})\) and \((H\), \(\{\beta_\sigma \colon G_\sigma \longrightarrow H\}_{\sigma \in \Sigma})\) to get, respectively, that \(g \cdot h = \lim_{\lambda \in \Lambda} g_{\eta_\lambda} \cdot h_{\eta_\lambda}\) and \(\beta(g) \cdot \beta(h) = \lim_{\lambda \in \Lambda} g_{\eta_\lambda}' \cdot h_{\eta_\lambda}'\).
But this means that the pair \((g \cdot h\), \(\beta(g) \cdot \beta(h))\), as the limit of \(\{(g_{\eta_\lambda} \cdot h_{\eta_\lambda}\), \(g_{\eta_\lambda}' \cdot h_{\eta_\lambda}')\}_{\lambda \in \Lambda}\), is an element of~\(\Gamma(\beta)\), so \(\beta\) is a homomorphism, what ends the proof of point (\ref{main_groups f1}) of the ``furthermore'' part of the proof.
The argument for (\ref{main_groups f2}) is the same as in the proof of Theorem \ref{main}.
\end{proof}

It might seem that the above reasoning would also work if we, in place of bi--invariant metrics, did consider left--invariant metrics on our compact metric groups. However, as the following example shows, this is not the case.

\begin{ex} \label{ex left-inv}
Consider the permutation group \(S_3\) of the set \(\{1\), \(2\), \(3\}\).
Set a mapping
\[
\varrho \colon S_3 \times S_3 \ni (\sigma, \,\tau) \longmapsto \varrho(\sigma, \, \tau) \in \mathbb{R}_{+},
\] where \(\varrho(\sigma\), \(\tau) = 0\), if \(\tau^{-1}\sigma \in \{\OPN{id}\), \((1 \, 2)\}\), and \(\varrho(\sigma\), \(\tau) = 0\) otherwise.
This mapping is clearly symmetric~and satisfies the triangle inequality.
Let \(\delta\) be the discrete metric on \(S_3\).
Define \(d_n \defeq \varrho + \frac{1}{n} \delta\).
Then the metric~\(d_n\) on \(S_3\) is left--invariant (and not bi--invariant), and \(d_{n + 1} \leqslant d_n\) for every \(n \in \mathbb{N}\).
Because of~that, the pair \((\{(S_3\), \(d_n)\}_{n \in \mathbb{N}}\), \(\{\OPN{id}_{S_3}\}_{n \in \mathbb{N}})\) is a direct system in the category of compact metric spaces with non--expansive homomorphisms as arrows, and its arrows are surjective.
If our approach from the above theorem would work in the case of groups with left--invariant metrics, what is easy to check, the object \((Z\), \(\zeta)\) of an inductive limit, as~the Gromov--Hausdorff limit of the sequence \((S_3\), \(d_n)\), would be the quotient of the pseudometric space \((S_3\), \(d)\) by \(\{d = 0\}\), where \(d\) is the pointwise limit of the sequence \(\{d_n\}_{n \in \mathbb{N}}\).
However, this cannot be~the~case, for there is no epimorphism from \(S_3\) onto~\(Z = \{\OPN{id}\),~\((1 \, 3)\),~\((2 \, 3)\}\).
(Surprisingly, what is interesting~and quite easy to check, the direct system we are considering does have an inductive limit with the trivial group as~its~object, see, e.g., the next section.)
\end{ex}

The above example shows, that if we want to show theorem on the existence of an inductive limit for~compact metric groups in general, we have to change or upgrade our approach.
Fortunately, we managed to~solve this problem, the results of which are presented in the next section.

\subsection{Inverse limit}
We now switch to the case of inverse systems of compact metric groups.
Similarly~as~before, we will first focus on compact metric groups with bi--invariant metrics.

\begin{thm} \label{inverse limit - main_bi_mg}
Let \((\{(G^\sigma\), \(d^\sigma)\}_{\sigma \in \Sigma}\), \(\{\alpha^{\sigma}_{\tau} \colon G^\tau\longrightarrow G^\sigma\}_{\sigma \leqslant \tau})\) be an inverse system in the category of compact metric groups with bi--invariant metrics with non--expansive homomorphisms as arrows, where \(\alpha^\sigma_\tau\) are~surjective for all \(\sigma \leqslant \tau\).
Assume that there exists an object in \(\mathsf{cmn}\) and surjective arrows from this object onto~\(G^\sigma\)
for \(\sigma \in \Sigma\) in \(\mathsf{cmn}\). Then this system has inverse limit. \\
Furthermore,
\begin{enumerate}
\item its object does~not depend, up to isometric isomorphism, on the choice of surjective arrows for~the~system, and \label{inverse limit - main_bi_mg f1}
\item its arrows are surjective. \label{inverse limit - main_bi_mg f2}
\end{enumerate}
\end{thm}

\begin{proof}
The proof will be held in a similar manner to the proof of Lemma \ref{inverse limit - main}.
We begin by embedding isometrically all the elements of the family \(\{G^{\sigma}\}_{\sigma \in \Sigma}\) in some compact metric space \((Y\), \(\varrho)\).
We consider \( \gamma_\tau = (G^\tau\),~\(\{\Gamma(\beta_{\tau}^{\sigma})\}_{\sigma \in \Sigma}) \in \mathfrak{K}(Y) \times \prod_{\sigma \in \Sigma} \mathfrak{K}(Y \times Y) \) for \(\tau \in \Sigma\), where
\(\mathfrak{K}(Y)\) and \(\mathfrak{K}(Y \times Y)\) are~the~hyperspaces of, respectively, \(Y\) and \(Y \times Y\) (considered with the sum metric \(\varrho_1\)), \(\Gamma(\beta_{\tau}^{\sigma})\) are~the~graphs of~mappings \(\beta_{\tau}^{\sigma}\) for~\(\sigma\), \(\tau \in \Sigma\), while~\(\beta_{\tau}^{\sigma} = \alpha_{\tau}^{\sigma}\) if \(\sigma \leqslant \tau\) and \(\beta_{\tau}^{\sigma} = \OPN{id}_{G^\tau}\) otherwise.
Finally, we use the~Tikhonov's theorem to~find a convergent subnet \( \{\gamma_{\eta_{\lambda}}\}_{\lambda \in \Lambda} \) of~\( \{\gamma_\tau\}_{\tau \in \Sigma}\) in \(\mathfrak{K}(Y) \times \prod_{\sigma \in \Sigma} \mathfrak{K}(Y \times Y)\).
Limit of this net consists of an object \(G \in \mathsf{cmn}\) together with~the~set of graphs of non--expansive (and surjective) arrows \(\{\alpha^\sigma \colon G \longrightarrow G^\sigma\}_{\sigma \in \Sigma}\) in~\(\mathsf{cmn}\) (we remember that these form an inverse limit of our system provided that everything happens in~\(\mathsf{cmn}\)). \linebreak
(We also remember the important property \eqref{eq 4} of these arrows.)
To finish the proof of the main part of this theorem, one need to show, in particular, that there is a group structure on~\(G\) such that the metric \(\varrho|_{G \times G}\) on \(G\) is bi--invariant, and the mappings \(\{\alpha^\sigma\}_{\sigma \in \Sigma}\) are homomorphisms.

We will now set a group structure on \(G\) by defining a binary operation on it.
Before this, thanks to the~Tikhonov's theorem, without loss of generality we assume that for every \(g\), \(h \in G\) sequences \(\{\alpha^{\eta_\lambda}(g) \cdot \alpha^{\eta_\lambda}(h)\}_{\lambda \in \Lambda}\) (for the convenience, we will use the notation of \(\cdot\) in the case of group operations for~all of~the~groups) and \(\{\alpha^{\eta_\lambda}(g)^{-1}\}_{\lambda \in \Lambda}\) are convergent.
We then define \(\cdot \colon G \times G \longrightarrow G\) in the~following~way:
\[
\displaystyle g \cdot h \defeq \lim_{\lambda \, \in \, \Lambda} \alpha^{\eta_\lambda}(g) \cdot \alpha^{\eta_\lambda}(h).
\]
To~see that for every \(g\), \(h \in G\) we have \(\alpha^\sigma(g) \cdot \alpha^\sigma(h) = \alpha^\sigma(g \cdot h)\), it is enough to write that
\[
\displaystyle \alpha^\sigma(g) \cdot \alpha^\sigma(h)
=
\displaystyle \alpha^\sigma_{\eta_\lambda}(\alpha^{\eta_\lambda}(g)) \cdot \alpha^\sigma_{\eta_\lambda}(\alpha^{\eta_\lambda}(h))
=
\displaystyle \alpha^\sigma_{\eta_\lambda}(\alpha^{\eta_\lambda}(g) \cdot \alpha^{\eta_\lambda}(h))
\]
for sufficiently large \(\lambda \in \Lambda\), and to look for a while at the above definition of our binary operation and property~\eqref{eq 4} of~\(\alpha^\sigma\).
Checking that this operation sets the group structure on \(G\) is a straightforward computation and is left as~an~easy exercise.
Of course, \(\varrho|_{G \times G}\) is bi--invariant, because every metric  \(d^{\eta_\lambda}\) for \(\lambda \in \Lambda\) is bi--invariant.

We will now finish the proof of the fact that the pair \((G\), \(\{\alpha^\sigma \colon G \longrightarrow G^\sigma\}_{\sigma \in \Sigma})\) is indeed an inverse limit of~our~system.
Assume then that there is an object \(H\) and a family of arrows \(\{\xi^\sigma \colon H \longrightarrow G^\sigma \}_{\sigma \in \Sigma}\) such~that \(\alpha_{\tau}^{\sigma} \circ \xi^{\tau} = \xi^{\sigma}\) for~\(\sigma \leqslant \tau\) in \(\Sigma\).
Without loss of generality we assume that \(\{\alpha^{\eta_\lambda}\}_{\lambda \in \Lambda}\) and \(\{\xi^{\eta_\lambda}\}_{\lambda \in \Lambda}\) are uniformly convergent, while the first one, what is already well known to us, converges to \(\OPN{id}_G\).
We set \(\nu = \lim_{\lambda \in \Lambda} \xi^{\eta_\lambda}\).
It remains to show that \(\nu \colon H \longrightarrow G\) is a homomorphism (we remember that \(\alpha^\sigma \circ \nu = \xi^\sigma\) for \(\sigma \in \Sigma\)).
But, by the definition of our binary operation, this~is the case, because for every \(\lambda \in \Lambda\) the mappings \(\alpha^{\eta_\lambda}\) and \(\xi^{\eta_\lambda}\) satisfy dependence \(\alpha^{\eta_\lambda} \circ \nu = \xi^{\eta_\lambda}\), both are~homomorphisms, and, finally, we have the aforementioned convergence of \(\{\alpha^{\eta_\lambda}\}_{\lambda \in \Lambda}\).

Point (\ref{inverse limit - main_bi_mg f1}) of the ``furthermore'' part is carried in the similar manner as the analogous part of the proof of~Theorem~\ref{main_groups}.
We will use the same property as before, this time, of our inverse limit, that is the fact that~in~the~situation as~in~the above proof, if for every \(\lambda \in \Lambda\) we have \(g^{\eta_\lambda}\), \(h^{\eta_\lambda} \in G^{\eta_\lambda}\) such that \(g^{\eta_\lambda} \overset{_{\lambda \in \Lambda}}{\longrightarrow} g\)~and \(h^{\eta_\lambda} \overset{_{\lambda \in \Lambda}}{\longrightarrow} h\) (here~both \(g\) and \(h\) are of course in \(G\)), then \(g^{\eta_\lambda} \cdot h^{\eta_\lambda} \overset{_{\lambda \in \Lambda}}{\longrightarrow} g \cdot h \in G\).
Note that, because \(\{\alpha^{\eta_\lambda}\}_{\lambda \in \Lambda}\) converges uniformly to~\(\OPN{id}_G\), we have \(g = \lim_{\lambda \in \Lambda} \alpha^{\eta_\lambda}(g)\) and \(h = \lim_{\lambda \in \Lambda} \alpha^{\eta_\lambda}(h)\).
Using this fact, definition of~our~binary operation and the fact that the metrics on our groups are bi--invariant, we can write
\[
\begin{array}{lll}
\displaystyle d^{\sigma}\left(g^{\eta_\lambda} \cdot h^{\eta_\lambda}, \, \alpha^{\eta_\lambda}(g) \cdot \alpha^{\eta_\lambda}(h)\right)
& \leqslant &
\displaystyle d^{\sigma}\left(g^{\eta_\lambda} \cdot h^{\eta_\lambda}, \, \alpha^{\eta_\lambda}(g) \cdot h^{\eta_\lambda}\right)
+
\displaystyle d^{\sigma}\left(\alpha^{\eta_\lambda}(g) \cdot h^{\eta_\lambda}, \, \alpha^{\eta_\lambda}(g) \cdot \alpha^{\eta_\lambda}(h)\right)\\
& = &
\displaystyle d^{\sigma}\left(g^{\eta_\lambda}, \, \alpha^{\eta_\lambda}(g)\right)
+
\displaystyle d^{\sigma}\left(h^{\eta_\lambda}, \, \alpha^{\eta_\lambda}(h)\right),
\end{array}
\]
what proves our property.

Let now \((H\), \(\{\beta^\sigma \colon H \longrightarrow G^\sigma\}_{\sigma \in \Sigma})\) be an inverse limit, constructed as in the main part of this proof, of~an~inverse system \((\{(G^\sigma, \, d^\sigma)\}_{\sigma \in \Sigma}\), \(\{\beta^{\sigma}_{\tau} \colon G^\tau \longrightarrow G^\sigma\}_{\sigma \leqslant \tau})\) satisfying assumptions of this theorem. \linebreak
We will find a non--expansive homomorphism from \(G\) onto \(H\).
To achieve that, note that, without loss of~generality, we can assume that the net \(\{\alpha^{\eta_\lambda}\}_{\lambda \in \Lambda}\) is uniformly convergent to a non--expansive surjection \(\alpha \colon G \longrightarrow H\).
What is left to show is that \(\alpha\) is a homomorphism.
Let then \(g\), \(h \in G\) and note, that, as~\(\alpha(g) = \lim_{\lambda \in \Lambda}\alpha^{\eta_\lambda}(g)\) and \(\alpha(h) = \lim_{\lambda \in \Lambda}\alpha^{\eta_\lambda}(h)\), the proven in the previous paragraph property of~the~limit \((H\), \(\{\beta^\sigma \colon H \longrightarrow G^\sigma\}_{\sigma \in \Sigma})\) clearly shows that \(\alpha(g) \cdot \alpha(h) = \lim_{\lambda \in \Lambda} \alpha^{\eta_\lambda}(g) \cdot \alpha^{\eta_\lambda}(h)\), i.e., that \(\alpha\) is~a~homomorphism, what ends the proof, as the proof of (\ref{inverse limit - main_bi_mg f2}) is the same as in Lemma \ref{inverse limit - main}.
\end{proof}

In the previous section, in case of theorem on the existence of an inverse limit for \(\mathsf{cmn}\), we presented an~equivalent condition for the existence of a limit.
Namely, an inverse limit existed exactly, when the set of~objects of a given system was bounded from above with respect to \(\preccurlyeq\).
However, in case of the category we~are~considering now -- this is no longer the case.
This time, as we have seen in Theorem \ref{inverse limit - main_bi_mg}, this condition ensures existence of an inverse limit, but is no longer a necessary condition for that.
We present below a~few examples, when this condition is not fulfilled, yet for some of them inverse limit exists, and for another it~does not exist.

\begin{ex} \label{ex inv_b_g 1}
Consider the circle group \(\mathbb{T}\) equipped with the Euclidian metric \(d_e\).
Define for \(n \in \mathbb{N}\) the~group \(\mathbb{T}_n = (\mathbb{T}\), \(n \cdot d_e)\).
Clearly, the family of objects \(\{\mathbb{T}_n\}_{n \in \mathbb{N}}\) is not bounded from above with respect to \(\preccurlyeq\). \linebreak
Let the mapping \(\alpha_n^m \colon \mathbb{T}_n \ni z \longmapsto z \in \mathbb{T}_m\) be given for \(m \leqslant n\) in \(\mathbb{N}\).
Then the pair \((\{\mathbb{T}_n\}_{n \in \mathbb{N}}\), \(\{\alpha_n^m\}_{m \leqslant n})\) fulfills all, but aforementioned one, of the assumption of Theorem \ref{inverse limit - main_bi_mg}.
Note that for every compact metric group \(G\) with bi--invariant metric there exists exaclty one set of arrows \(\{\alpha^n \colon G \longrightarrow \mathbb{T}_n\}_{n \in \mathbb{N}}\) that is compatible with the set of arrows of our system, i.e., \(\alpha^n\) must be a trivial homomorphism for every \(n\).
Thanks to~this~fact, it is easy to see that the given inverse system has an inverse limit with trivial group as its object.
What is more, the arrows of this limit are not surjective.

This example shows that the idea of the surjectivity of arrows of an inverse system presented in Lemma~\(\ref{inverse limit - main}\) cannot, in general, be transferred to a wider context.
\end{ex}

\begin{ex} \label{ex inv_b_g 2}
Let \(\mathbb{Z}_n\) for \(n \in \mathbb{N}\) be the additive group of integers modulo \(n\) considered with discrete metric.
Consider inverse system \((\{\mathbb{Z}_{2^n}\}_{n \in \mathbb{N}}\), \(\{\alpha_n^m \colon \mathbb{Z}_{2^n} \longrightarrow \mathbb{Z}_{2^m}\}_{m \leqslant n})\), where \(\alpha_{n + 1}^n\) for \(n \in \mathbb{N}\) and \(k \in \mathbb{Z}_n\) is defined in the following way:
\[
\displaystyle \alpha_{n + 1}^n(k)
\defeq
\left\{
\begin{array}{ll}
\frac{k}{2}, & \textnormal{ if } k \textnormal{ is even},\\
\displaystyle k, & \textnormal{ if } k \textnormal{ is odd},
\end{array}\right. 
\]
and the rest of arrows is defined in an obvious way.
We will justify, similarly as in the relevant part of~the~proof of Lemma \ref{inverse limit - main}, that this system has no inverse limit.
Fix an object \(G\) and a set of arrows \(\{\alpha^n \colon G \longrightarrow \mathbb{Z}_n\}_{n \in \mathbb{N}}\) in our category such that the condition (\ref{Inverse limit 1}) from Definition~\ref{Inverse limit} holds.
Of course, there is no upper bound with respect to \(\preccurlyeq\) of objects of our system, so there is \(n_0 \in \mathbb{N}\) and \(k_0 \in \mathbb{Z}_{2^{n_0}}\) for~which~\(k_0 \notin \alpha^{n_0}(G)\). \linebreak
Let \(H\) be a cyclic subgroup \(\langle k_0 \rangle\) of \(\mathbb{Z}_{2^{n_0}}\).
Define \(\xi^{n_0} \colon H \ni k \longrightarrow k \in \mathbb{Z}_{2^{n_0}}\) and note that this arrow determines for every \(n \neq n_0\) exactly one arrow \(\xi^n \colon H \longrightarrow \mathbb{Z}_{2^n}\) such that \(\{\xi^n\}_{n \in \mathbb{N}}\) coincides with the set of arrows of~the~given system.
But, finally, with this configuration there is no \(\nu \colon H \longrightarrow G\) such that \(\alpha^{n_0}(\nu(k_0)) = \xi^{n_0}(k_0)\), so our system has no inverse limit.
\end{ex}

Using the above examples and some easy trick, we will produce for every inverse system (indexed by \(\mathbb{N}\)) satisfying assumptions of Theorem \ref{inverse limit - main_bi_mg} two unbounded (with respect to \(\preccurlyeq\)) inverse systems -- one with, and one without an inverse limit, while this existing one will have the same object as the object of the limit of~original system.

\begin{ex} \label{ex inv_b_g 3}
Let \((\{(G^n\), \(d^n)\}_{n \in \mathbb{N}}\), \(\{\gamma^{m}_{n} \colon G^n \longrightarrow G^m\}_{m \leqslant n})\) be an inverse system satisfying assumptions of~Theorem \ref{inverse limit - main_bi_mg} and \((G\),~\(\{\gamma^n \colon G \longrightarrow G^n\}_{n \in \mathbb{N}})\) be its inverse limit.
Let \((\{\mathbb{T}_{n}\}_{n \in \mathbb{N}}\), \(\{\alpha_n^m \colon \mathbb{T}_{n} \longrightarrow \mathbb{T}_{m}\}_{m \leqslant n})\) be~an~inverse system as in Example~\ref{ex inv_b_g 1}.
For every \(n \in \mathbb{N}\) consider \(H^n = \mathbb{T}_{2^n} \times G^n\) with sum metric.
The~set \(\{H^n\}_{n \in \mathbb{N}}\) is not bounded from~above with respect to \(\preccurlyeq\).
Let \(\beta_n^m \colon H^n \ni (z\), \(g) \longmapsto (\alpha_n^m(z)\),~\(\gamma_n^m(g)) \in H^m\) for all \(m \leqslant n\) in \(\mathbb{N}\).
In this situation, the pair \((\{H^n\}_{n \in \mathbb{N}}\), \(\{\beta^{m}_{n} \colon H^n \longrightarrow H^m\}_{m \leqslant n})\) is an inverse system (with surjective arrows) in the category under consideration, and its objects are not bounded from above with respect to \(\preccurlyeq\).
However, it is straightforward to check that the pair \((G\), \(\{\beta^n \colon G \longrightarrow H^n\}_{n \in \mathbb{N}})\), where \(\beta^n \colon G \ni g \longmapsto (\alpha^n(z)\), \(\gamma^n(g)) \in H^n\) for every \(n \in \mathbb{N}\), is its inverse limit.
\end{ex}

\begin{ex}
Similarly as in the previous example, let \((\{(G^n\), \(d^n)\}_{n \in \mathbb{N}}\), \(\{\gamma^{m}_{n} \colon G^n \longrightarrow G^m\}_{m \leqslant n})\) be~an~inverse system satisfying assumptions of~Theorem \ref{inverse limit - main_bi_mg} with its inverse limit \((G\),~\(\{\gamma^n \colon G \longrightarrow G^n\}_{n \in \mathbb{N}})\), and let
\((\{\mathbb{Z}_{2^n}\}_{n \in \mathbb{N}}\), \(\{\alpha_n^m \colon \mathbb{Z}_{2^n} \longrightarrow \mathbb{Z}_{2^m}\}_{m \leqslant n})\) be an inverse system as in Example \ref{ex inv_b_g 2}.
One can show, similarly as in the aforementioned example, that the inverse system \((\{H^n\}_{n \in \mathbb{N}}\), \(\{\beta^{m}_{n} \colon H^n \longrightarrow H^m\}_{m \leqslant n})\), where \(H^n = \mathbb{Z}_{2^{n}} \times G^n\) for \(n \in \mathbb{N}\) is considered with sum metric and \(\beta_n^m \colon H^n \ni (z\), \(g) \longmapsto (\alpha_n^m(z)\),~\(\gamma_n^m(g)) \in H^m\) for all \(m \leqslant n\) in \(\mathbb{N}\), has no inverse limit.
\end{ex}

\begin{note}
It is worth noting that if we additionally required from the category considered in this section that~its~objects are abelian groups, then the proofs of Theorems \ref{main_groups} and \ref{inverse limit - main_bi_mg} would work in this case as well.
\end{note}

\section{Compact metric groups} \label{Section 4}
Finally, we come to proving our theorems for compact metric groups in general.
In case of both theorems, we will use the following notion.
If \(G\) is a compact topological group and \(d\) is a left--invariant pseudometric on \(G\) that induces its topology, we define
\[
\displaystyle \widehat{d}(g_1, \, g_2) \defeq \sup_{h \, \in \, G} d(g_1 \cdot h, \, g_2 \cdot h)
\]
for \(g_1\), \(g_2 \in G\).
It is straightforward to check that \(\widehat{d}\) is the smallest bi--invariant pseudometric that is not smaller than (and equivalent to) \(d\).
In particular, \(\widehat{G} = (G\), \(\widehat{d})\) is a compact metric group with bi--invariant metric.

Our proofs will rely on the following fact regarding this notion.

\begin{lem} \label{hat lemma}
Let \(\varphi \colon G \longrightarrow H\) be a non--expansive and surjective mapping between compact metric groups \((G\), \(d_G)\) and \((H\), \(d_H)\).
Then \(\widehat{\varphi} \colon \widehat{G} \ni g \longmapsto \varphi(g) \in \widehat{H}\) is a non--expansive surjection.
\end{lem}

\begin{proof}
Since \(\varphi\) is non--expansive, we have \(d_H \circ (\varphi \times \varphi)\) is a left--invariant pseudometric on \(G\) that is not greater than~\(d_G\).
But \(d_G \leqslant \widehat{d_G}\), so \(\widehat{d_H \circ (\varphi \times \varphi)} \leqslant \widehat{d_G}\).
However, because \(\varphi\) is a surjective homomorphism, one can easly see that \(\widehat{d_H \circ (\varphi \times \varphi)} = \widehat{d_H} \circ (\varphi \times \varphi)\), which means that \(\widehat{\varphi}\) is non--expansive.
\end{proof}

\subsection*{Inductive limit}
We now move to proove the theorem on the existence of an inductive limit in case of~a~direct system of compact metric groups.

\begin{thm} \label{main_inductive_compact_metric_groups}
Let \((\{(G_\sigma\), \(d_\sigma)\}_{\sigma \in \Sigma}\), \(\{\alpha_{\sigma}^{\tau} \colon G_\sigma \longrightarrow G_\tau\}_{\sigma \leqslant \tau})\) be a direct system in the category of compact metric groups with non--expansive homomorphisms as arrows, where \(\alpha_\sigma^\tau\) are~surjective for all \(\sigma \leqslant \tau\). \linebreak
Then this system has inductive limit.\\
Furthermore,
\begin{enumerate}
\item its object does not depend, up~to~isometric isomorphism, on the choice of surjective arrows for~the~system, and \label{main_inductive_compact_metric_groups f1}
\item its arrows are surjective. \label{main_inductive_compact_metric_groups f3}
\end{enumerate}
\end{thm}

\begin{proof}
Using Lemma \ref{hat lemma} it is clear that \((\{\widehat{G_\sigma}\}_{\sigma \in \Sigma}\), \(\{\widehat{\alpha_{\sigma}^{\tau}}\}_{\sigma \leqslant \tau})\) is a direct system satisfying assumtions of~Theorem \ref{main_groups}.
Therefore, it has an inductive limit \(((H\), \(d_H)\), \(\{\widehat{\alpha_\sigma} \colon \widehat{G_\sigma} \longrightarrow H\})\).
Let \(d\) be the greatest left--invariant pseudometric on \(H\) such that \(d \leqslant d_H\) and each \(\alpha_\sigma \colon G_\sigma \longrightarrow H\) for \(\sigma \in \Sigma\) is non--expansive with respect to~\(d\).
Let \(N\) be~the~smallest closed normal subgroup of \(H\) containing the set \(\{h: \ d(h\),~\(e_H) = 0\}\).
Consider \(G\) to be the quotient \(\faktor{G}{N}\) with the quotient map \(\pi \colon H \longrightarrow G\).
Of course, \(G\)~is a compact topological group.
Define 
\[
\displaystyle d_G \colon G \times G \ni ([g]_N, \, [h]_N) \longmapsto \inf_{n_1, \, n_2 \, \in \, N} d(g \, n_1, \, h \, n_2) \in \mathbb{R}_+.
\]

We now justify that \(d_G\) is a left--invariant metric on \(G\).
Let \(g\), \(h \in H\) be such that \(d_G([g]_N\),~\([h]_N) = 0\).
This means that there is a sequence \(\{n_j\}_{j \in \mathbb{N}}\) in \([h^{-1} \, g]_H\) such that \(d(h^{-1} \, g \, n_j\),~\(e_H) \longrightarrow 0\) when \(j \longrightarrow \infty\).
However, by the compactness of \(H\), without loss of generality we may and do assume that~\(\{h^{-1} \, g \, n_j\}_{j \in \mathbb{N}}\) converges to \(f \in [h^{-1} \, g]_N\).
But then \(d(f\), \(e_H) = 0\), so \([h^{-1} \, g]_N = N\), and thus \(d_G\) separates points. \linebreak
To see that \(d_G\) is a metric, it is enough to check the triangle inequality.
Fix \(f\), \(g\), \(h \in H\).
Pseudometric \(d\) is~bi--invariant, so
\[
\displaystyle d(f \, n_1, \ h \, n_3)
\leqslant
\displaystyle d(g^{-1} \, f \, n_1 \, n_2^{-1}, \ e_H)
+
\displaystyle d(h^{-1} \, g \, n_2 \, n_3^{-1}, \ e_H)\\
\]
for every \(n_1\), \(n_2\), \(n_3 \in N\).
But left and right multiplications on a group \(N\) are bijective, therefore, after~taking~infimum  over \(n_1\), \(n_2\), \(n_3\) on both sides of the above inequality we get the triangle inequality.
Finally, it~is~obvious that \(d_G\) is bi--invariant.

We now show that \(d_G\) induces topology of \(G\).
Let \(\{g_j\}_{j \in \mathbb{N}}\) be a sequence in \(G\) convergent to some \(g\). \linebreak
We find for every \(j \in \mathbb{N}\) an element \(h_j \in [g_j]_N\).
By the compactness of \(H\), there is a subsequence \(\{h_{\nu_j}\}_{j \in \mathbb{N}}\) of~\(\{h_j\}_{j \in \mathbb{N}}\) that converges to \(h \in H\) with respect to \(d_H\).
Because of that \(d_G([h_{\nu_j}]_N\), \([h]_N) \leqslant d_H(h_{\nu_j}\), \(h) \longrightarrow 0\), when \(j \longrightarrow \infty\), so \(\{[h_{\nu_j}]_N\}_{j \in \mathbb{N}}\) converges to \([h]_N\) with respect to \(d_G\).
But the quotient map \(\pi \colon H \longrightarrow G\) is~continuous, so \([g]_N = [h]_N\).
Starting this reasoning with an arbitrary subsequence of \(\{g_j\}_{j \in \mathbb{N}}\) we would get that \(G \ni g \longmapsto g \in (G\), \(d_G)\) is continuous.
Similar argument shows that this mapping is a homeomorphism, so \(d_G\) indeed induces topology of \(G\).

We have shown that \(d_G\) is a left--invariant metric on \(G = \faktor{H}{N}\) that induces its topology.
For every \(\sigma \in \Sigma\) consider a surjecive arrow \(\beta_\sigma \defeq~\pi \circ \alpha_\sigma\).
Of course, the family \(\{\beta_\sigma\}_{\sigma \in \Sigma}\) fulfills condition~(\ref{Inductive limit 1}) from~Definition~\ref{Inductive limit}.
We will prove that the pair \(((G\), \(d_G)\), \(\{\beta_\sigma \colon G_\sigma \longrightarrow G\}_{\sigma \in \Sigma})\) is an inducive limit of~our~direct system.

Assume that \((K\), \(d_K)\) is a compact metric group and \(\{\xi_\sigma \colon G_\sigma \longrightarrow K\}_{\sigma \in \Sigma}\) is a family of non--expansive homomorphisms such that \(\xi_\tau \circ \alpha_\sigma^\tau = \xi_\sigma\) for \(\sigma \leqslant \tau\) in \(\Sigma\).
Then the family \(\{\widehat{\xi_\sigma} \colon \widehat{G_\sigma} \longrightarrow \widehat{K}\}_{\sigma \in \Sigma}\) has the~analogous property.
Therefore, because \(H\) is an object of a direct limit of \((\{\widehat{G_\sigma}\}_{\sigma \in \Sigma}\), \(\{\widehat{\alpha_\sigma^\tau}\}_{\sigma \leqslant \tau})\) in the category of~compact metric groups with bi--inariant metrics, there exists a unique \(\nu \colon H \longrightarrow \widehat{K}\) such that \(\nu \circ \widehat{\alpha_\sigma} = \widehat{\xi_\sigma}\) for every \(\sigma \in \Sigma\).
But \(\nu\) is non--expansive, that is \(\widehat{d_K} \circ (\nu \times \nu) \leqslant d_H\), so all the more \(d_K \circ (\nu \times \nu) \leqslant d_H\), as~well~as \(\xi_\sigma\) for every \(\sigma \in \Sigma\), and thus
\[
\displaystyle (d_K \circ (\nu \times \nu)) \circ (\alpha_\sigma \times \alpha_\sigma)
=
\displaystyle d_K \circ (\xi_\sigma \times \xi_\sigma)
\leqslant
\displaystyle d_\sigma
\]
for every \(\sigma \in \Sigma\).
By the choice of \(d\) we get that \(d_K \circ (\nu \times \nu) \leqslant d\).
This means that \(\ker(\nu)\) contains \(\{h: \ d(h\),~\(e_H) = 0\}\), so \(\ker(\nu)\) also contains \(N\).
Therefore, because \(N = \ker(\pi)\), there~exists an epimorphism \(\omega \colon G \longrightarrow K\) such that \(\nu = \omega \circ \pi\).
For every \(\sigma \in \Sigma\) we have then
\[
\displaystyle \xi_\sigma 
=
\displaystyle \nu \circ \alpha_\sigma
=
\displaystyle \omega \circ (\pi \circ \alpha_\sigma)
=
\displaystyle \omega \circ \beta_\sigma,
\]
and \(\omega\) is the only mapping with that property, for \(\beta_\sigma\) is surjective.
To finish the proof of the main part of~this~theorem, it is sufficient to show that \(\omega\) is non--expansive.
To accomplish this, we once again use the~fact that \(d_K \circ (\nu \times \nu) \leqslant~d\), i.e. \(d_K \circ ((\omega \times \omega) \circ (\pi \times \pi)) \leqslant d\), as this inequality is equivalent to~the~following one \(d_K \circ (\omega \times \omega) \leqslant d_G\) (to see this, it suffices to fix \(g\), \(h \in H\), write the first one with arguments \(g' \in [g]_N\), \(h' \in [h]_N\), and take infimum over such \(g'\), \(h'\) on both sides of the inequality).
This proves that \(\omega\) is non--expansive, and hence finishes the main part of the proof.

We now move to the ``furthermore'' part.
Let \(\sigma_0\) and \(Y\) be as in the proofs of previous versions of~this~theorem.
Let \((\{(G_\sigma, \, d_\sigma)\}_{\sigma \in \Sigma}\), \(\{\beta_{\sigma}^{\tau} \colon G_\sigma \longrightarrow G_\tau\}_{\sigma \leqslant \tau})\) be a direct system satisfying assumptions of this theorem.
Let \(((K\), \(d_K)\), \(\{\widehat{\beta_\sigma} \colon \widehat{G_\sigma} \longrightarrow K\}_{\sigma \in \Sigma})\) be an inductive limit (in the category of compact metric groups with bi--invariant metrics) of \((\{\widehat{G_\sigma}\}_{\sigma \in \Sigma}\), \(\{\widehat{\beta_{\sigma}^{\tau}}\}_{\sigma \leqslant \tau})\).
Finally, let \(L\) be an object of an inductive limit \linebreak (that was constructed in accordance with the previous part of this proof) of \((\{G_\sigma\}_{\sigma \in \Sigma}\),~\(\{\beta_{\sigma}^{\tau}\}_{\sigma \leqslant \tau})\).
To prove that there is an isometric isomorphism between \(G\) and \(L\), it is sufficient to find an isometric isomorphism \(\beta \colon (H\), \(d) \longrightarrow (K\), \(\zeta)\), where \(\zeta\) is the greatest left--invariant pseudometric on \(K\) such that \(\zeta \leqslant d_K\) and each~element of \(\{\beta_\sigma \colon G_\sigma \longrightarrow K\}_{\sigma \in \Sigma}\) is non--expansive with respect to \(\zeta\).
(Note that, because \(d_\tau \circ (\beta_\sigma^\tau \times \beta_\sigma^\tau) \leqslant d_\sigma\) for~all \(\sigma \leqslant \tau\) in \(\Sigma\), we can write a simplified equivalent version of the aforementioned definition of \(\zeta\), in~which we consider any subnet of \(\{\beta_\sigma\}_{\sigma \in \Sigma}\) instead of the whole net.
The same stands for \(d\).
We will stick to~this~simplified version.)
To achieve this, we will slightly modify the proof of the analogous part of~Theorem \ref{main_groups}.

Note that for every \(\sigma \geqslant \sigma_0\) there is a unique isometric isomorphism \(\widehat{\iota_\sigma} \colon \widehat{G_\sigma} \longrightarrow \widehat{G_\sigma}\) such that \(\widehat{\iota_\sigma} \circ \widehat{\alpha_{\sigma_0}^\sigma} = \widehat{\beta_{\sigma_0}^\sigma}\).
This in turn means that 
\begin{equation} \label{eq 5}
\displaystyle \widehat{\beta_\sigma} \circ \widehat{\iota_\sigma} \circ \widehat{\alpha_{\sigma_0}^\sigma} = \widehat{\beta_{\sigma_0}}
\end{equation}
for every \(\sigma \geqslant \sigma_0\).
Naturally, this time, instead of family \(\{\widehat{\beta_\sigma}\}_{\sigma \geqslant \sigma_0}\), we will consider \(\{\widehat{\beta_\sigma} \circ \widehat{\iota_\sigma}\}_{\sigma \geqslant \sigma_0}\).
Similarly as in Theorem \ref{main_groups}, we infer that there exists a subnet \(\{\Gamma(\widehat{\beta_{\eta_\lambda}} \circ \widehat{\iota_{\eta_\lambda}})\}_{\lambda \in \Lambda}\) of \(\{\Gamma(\widehat{\beta_\sigma} \circ \widehat{\iota_\sigma})\}_{\sigma \geqslant \sigma_0}\) that~is~convergent in~\(\mathfrak{K}(Y \times Y)\) to a graph of an isometric isomorphism \(\beta \colon H \longrightarrow K\).
Fix \(g \in G_{\sigma_0}\) and consider the net \(\{\widehat{\alpha_{\sigma_0}^{\eta_\lambda}}(g)\}_{\lambda \in \Lambda}\).
We are already well aware that this net converges to \(\widehat{\alpha_{\sigma_0}}(g)\).
Therefore, because \(\widehat{\beta_{\eta_\lambda}}(\widehat{\iota_{\eta_\lambda}}(\widehat{\alpha_{\sigma_0}^{\eta_\lambda}}(g))) = \widehat{\beta_{\sigma_0}}(g)\) for every \(\lambda \in \Lambda\), we get that 
\[
\displaystyle \Gamma\left(\widehat{\beta_{\eta_\lambda}} \circ \widehat{\iota_{\eta_\lambda}}\right) \ni \left(\widehat{\alpha_{\sigma_0}^{\eta_\lambda}}(g), \ \widehat{\beta_{\eta_\lambda}}\left(\widehat{\iota_{\eta_\lambda}}\left(\widehat{\alpha_{\sigma_0}^{\eta_\lambda}}(g)\right)\right)\right)
\overset{\lambda \in \Lambda}{\longrightarrow}
\displaystyle \left(\widehat{\alpha_{\sigma_0}}(g), \ \widehat{\beta_{\sigma_0}}(g)\right) \in \Gamma(\beta).
\]
This means that \(\beta(\widehat{\alpha_{\sigma_0}}(g)) = \widehat{\beta_{\sigma_0}}(g)\).
But this fact, together with \eqref{eq 5}, allows us to write for \(\sigma \geqslant \sigma_0\)
\[
\displaystyle \beta\left(\widehat{\alpha_{\sigma}}\left(\widehat{\alpha_{\sigma_0}^\sigma}(g)\right)\right)
=
\displaystyle \beta\left(\widehat{\alpha_{\sigma_0}}(g)\right)
=
\displaystyle \widehat{\beta_{\sigma_0}}(g)
=
\displaystyle \widehat{\beta_\sigma}\left(\widehat{\iota_{\sigma}}\left(\widehat{\alpha_{\sigma_0}^\sigma}(g)\right)\right).
\]
Therefore, because \(\alpha_{\sigma_0}^\sigma\) is surjective, we have just shown that \(\beta \circ \widehat{\alpha_\sigma} = \widehat{\beta_\sigma} \circ \widehat{\iota_\sigma}\), i.e., 
\begin{equation} \label{eq 6}
\displaystyle \widehat{\alpha_\sigma}
=
\displaystyle \beta^{-1} \circ \widehat{\beta_\sigma} \circ \widehat{\iota_\sigma},
\end{equation}
for~every \(\sigma \geqslant \sigma_0\) (here, of course, \(\iota_{\sigma_0} = \OPN{id}_{G_{\sigma_0}}\)).

Let \(\delta_H\) be a left--invariant pseudometric on \(H\) such that \(\delta_H \leqslant d_H\) and every element of the family \(\{\alpha_\sigma \colon G_\sigma \longrightarrow H\}_{\sigma \geqslant \sigma_0}\) is non--expansive with respect to \(\delta_H\).
Thanks to this and \eqref{eq 6}, it is straightforward to~check that \(\delta_K \defeq~\delta_H \circ (\beta^{-1} \times \beta^{-1})\) has analogous properties.
What is more, obviously, \(\beta \colon (H\),~\(\delta_H) \longrightarrow~(K\),~\(\delta_K)\) is~an~isometry.
But \(\beta\) and \(\widehat{\iota_\sigma}\) for \(\sigma \geqslant \sigma_0\) are isomorphisms, so we can write
\[
\displaystyle \widehat{\beta_\sigma}
=
\displaystyle \beta \circ \widehat{\alpha_\sigma} \circ \widehat{\iota_\sigma^{-1}}
\]
and repeat this reasoning in the opposite direction to get that \(\beta \colon (H\), \(d) \longrightarrow (K\), \(\zeta)\) is an isometric isomorphism, what was to prove.
Of course, one~proves the point (\ref{main_inductive_compact_metric_groups f3}) of the ``furthermore'' part exactly as~in~the~proof of~Theorem~\ref{main}.
\end{proof}

\subsection*{Inverse limit}
We will now make the statement of the last theorem concerning categorical limits of~this~paper, the proof of which is carried out in the same way as the proof of Theorem~\ref{inverse limit - main_bi_mg}.

\begin{thm} \label{main_inverse_compact_metric_groups}
Let \((\{(G^\sigma\), \(d^\sigma)\}_{\sigma \in \Sigma}\), \(\{\alpha^{\sigma}_{\tau} \colon G^\tau\longrightarrow G^\sigma\}_{\sigma \leqslant \tau})\) be an inverse system in the category of compact metric groups with non--expansive homomorphisms as arrows, where \(\alpha^\sigma_\tau\) are~surjective for all \(\sigma \leqslant \tau\). \linebreak
Assume that there exists an object in \(\mathsf{cmn}\) and surjective arrows from this object onto~\(G^\sigma\)
for \(\sigma \in \Sigma\) in \(\mathsf{cmn}\). \nopagebreak \linebreak
Then this system has inverse limit. \\
Furthermore,
\begin{enumerate}
\item its object does~not depend, up to isometric isomorphism, on the choice of surjective arrows for~the~system, and \label{main_inverse_compact_metric_groups f1}
\item its arrows are surjective. \label{main_inverse_compact_metric_groups f3}
\end{enumerate}
\end{thm}

\section{Iso--derivative and iso--height} \label{Section 5}
This section will be devoted to the presentation of the proposition of application of Theorem \ref{main}. \linebreak
Let \((X\), \(d)\) be a compact metric space.
Consider \(\OPN{Iso}(X\), \(d)\), the set of all isometries on \(X\), and equivalence relation \(\underset{^{\OPN{iso}}}{\sim}\)
on \(X\) defined as follows:
\[
x \underset{^{\OPN{iso}}}{\sim} y \overset{\OPN{def}}{\Longleftrightarrow} \varphi(x) = y \textnormal{ for some } \varphi \in \OPN{Iso}(X, \, d).
\]

\begin{deff}
Let \(\puo{X}{1}\) be the quotient set 
\(\faktor{X}{_{\underset{\OPN{iso}}{\sim}}}\).
Denote by \(\puo{d}{1}\) the greatest pseudometric on \(\puo{X}{1}\) that~makes the canonical projection \(\puo{\pi}{1}_X \colon (X\), \(d) \longrightarrow (\puo{X}{1}\), \(\puo{d}{1})\) non--expansive.
We call the space \((\puo{X}{1}\),~\(\puo{d}{1})\) to be the \emph{first iso--derivative} of~\((X\),~\(d)\) or, briefly, the \emph{iso--derivative} of \((X\), \(d)\).
If \((X\), \(d)\) is~such~that \((\puo{X}{1}\), \(\puo{d}{1}) = (X\), \(d)\), then we call it to be \emph{iso--rigid}.
\end{deff}

As was shown in \cite[Proposition \(5.1\)]{Central Points}, defined above \(\puo{d}{1}\) is a metric,
and for all \(x\), \(y \in X\)
\[
\displaystyle \puo{d}{1}\left(\puo{\pi}{1}_X(x), \, \puo{\pi}{1}_X(y)\right)
\displaystyle = \sup\left\{ \left|f\left(\puo{\pi}{1}_X(x)\right) - f\left(\puo{\pi}{1}_X(y)\right)\right| \colon \
\displaystyle \, f \colon \puo{X}{1} \rightarrow \mathbb{R}, \, f \circ \puo{\pi}{1}_X \textnormal{ is non--expansive}  \right\}.
\]

By this fact, the space \((\puo{X}{1}\), \(\puo{d}{1})\) is a compact metric space.
Fortunately, the formula for \(\puo{d}{1}\) can be expressed in even simpler and more explicit way.

\begin{prop} \label{d1}
For any \(x, y \in X\)
\[
\displaystyle
\puo{d}{1}\left(\puo{\pi}{1}_X(x), \, \puo{\pi}{1}_X(y)\right)
\ = \inf_{\varphi, \, \psi \, \in \, \OPN{Iso}(X, \, d)}d(\varphi(x), \, \psi(y)).
\]
\end{prop}

\begin{proof}
It is enough to prove that the above formula defines the greatest pseudometric for which \(\puo{\pi}{1}_X\) is~nonexpansive.
The only thing worth commenting on here is the triangle inequality.
Notice that for any \(x\), \(y \in X\) we have
\[
\displaystyle \inf_{\varphi, \, \psi \, \in \, \OPN{Iso}(X, \, d)}d(\varphi(x), \, \psi(y))
\ = \inf_{\zeta \, \in \, \OPN{Iso}(X, \, d)}d(\zeta(x), \, y),
\]
because \(d(\varphi(x)\), \(\psi(y)) = d((\psi^{-1} \circ \varphi)(x)\), \(y)\) for any \(\varphi\), \(\psi \in \OPN{Iso}(X, \, d)\).

But, for any \(x\), \(y\), \(z \in X\): 
\[
\displaystyle \inf_{\varphi, \, \psi \, \in \, \OPN{Iso}(X, \, d)}d(\varphi(x), \, \psi(z))
\displaystyle \leqslant \inf_{\varphi \, \in \, \OPN{Iso}(X, \, d)}d(\varphi(x), \, y) \ + \inf_{\psi \, \in \, \OPN{Iso}(X, \, d)}d(y, \, \psi(z)),
\]
which completes the proof.
\end{proof}

Basing on Theorem \ref{main} and transfinite induction, for every compact metric space we are able to define its iso--derivative of order \(\alpha\) for any ordinal number \(\alpha\).
To do so, we will have to define at the same~time the~canonical projections between our derivatives.
As a result, we will eventually get the transfinite sequence of isometric derivatives of a compact metric space together with canonical projections  associated~with those~derivatives.
In particular, both sets -- of derivatives and of canonical projections -- will together form a direct system in \(\mathsf{cmn}\) indexed by ordinal numbers.

\begin{deff}
The transfinite sequence of \emph{iso--derivatives} \(\{(\puo{X}{\alpha}\),~\(\puo{d}{\alpha})\}_{\alpha \geqslant 0}\), or \(\{(X\), \(d)^{(\alpha)}\}_{\alpha \geqslant 0}\), of~a~compact metric space \((X\),~\(d)\), and the set of canonical projections \(\{\pi_{\beta}^{\gamma} \colon \puo{X}{\beta} \longrightarrow~\puo{X}{\gamma}\}_{\beta\leqslant \gamma \leqslant \alpha}\) are defined in~the~way that follows.
\begin{enumerate}
\item If \(\alpha = 0\), then  \((X\), \(d)^{(0)} \defeq (X\), \(d)\) and \(\pi_0^0 \defeq \OPN{id}_X\).
\end{enumerate}

Assume now that we have already defined iso--derivatives \(\{(\puo{X}{\beta}\), \(\puo{d}{\beta})\}_{\beta \leqslant \alpha}\) and canonical projections \(\{\pi_{\beta}^{\gamma} \colon \puo{X}{\beta} \longrightarrow \puo{X}{\gamma}\}_{\beta \leqslant \gamma\leqslant \alpha}\) for some ordinal \(\alpha\).
\begin{enumerate}
\setItemnumber{2}
\item We set \((X\), \(d)^{(\alpha + 1)} \defeq (\puo{X}{\alpha}\), \(\puo{d}{\alpha})^{(1)}\) together with \(\pi_{\beta}^{\alpha + 1} \defeq \puo{\pi_{\puo{X}{\alpha}}}{1} \circ \pi_{\beta}^{\alpha}\) for \(\beta \leqslant \alpha\).
Of course, we also set \(\pi_{\alpha + 1}^{\alpha + 1} \defeq \OPN{id}_{\puo{X}{\alpha + 1}}\).
\end{enumerate}

Let \(\alpha\) be a limit ordinal.
Assume that we have already managed to define iso--derivatives \(\{(\puo{X}{\beta}\), \(\puo{d}{\beta})\}_{\beta < \alpha}\) together with projections \(\{\pi_{\beta}^{\gamma} \colon \puo{X}{\beta} \longrightarrow \puo{X}{\gamma}\}_{\beta \leqslant \gamma < \alpha}\).

\begin{enumerate}
\setItemnumber{3}
\item \label{iso-derivative direct system 3} Basing on Theorem \ref{main}, we define the pair \(((X\), \(d)^{(\alpha)}\), \(\{\pi_{\beta}^{\alpha} \colon \puo{X}{\beta} \longrightarrow \puo{X}{\alpha}\}_{\beta < \alpha})\) to be the~inductive limit of~the~direct system \((\{(\puo{X}{\beta}\), \(\puo{d}{\beta})\}_{\beta < \alpha}\), \(\{\pi_{\beta}^{\gamma}\}_{\beta \leqslant \gamma < \alpha})\),
and we define \(\pi_{\alpha}^{\alpha} \defeq \OPN{id}_{\puo{X}{\alpha}}\). 
\end{enumerate}
\end{deff}

\begin{note}
It is obvious that all the mappings occuring in the above definition are surjective morphisms in~\(\mathsf{cmn}\).
Simultaneously, by the way they are defined, in case of the first limit ordinal \(\omega\), we indeed deal in (\ref{iso-derivative direct system 3}) with~the~direct system which fulfills the assumptions of the Theorem \ref{main}.
In addition to that, the same theorem guarantees us that, in case of any other limit ordinal \(\alpha\), we always have in (\ref{iso-derivative direct system 3}) the direct system which fulfills the assumptions of the Theorem \ref{main}.
Therefore, iso--derivative of any order is well defined, and the transfinite sequence of iso--derivatives together with all the canonical mappings form a direct system in~the~category \(\mathsf{cmn}\).
In particular, what was established in the proof of Theorem \ref{main} and is worth to highlight is~that for every limit ordinal \(\alpha\) and \(X \in \mathfrak{GH}\), the iso--derivative \(\puo{X}{\alpha}\) is in fact a Gromov--Hausdorff limit of~the~sequence \(\{\puo{X}{\beta}\}_{\beta < \alpha}\).

For any limit ordinal \(\alpha\), the space \(\puo{X}{\alpha}\) is not naturally or canonically defined as set--theoretic object, but~it~is selected up to isometry.
Nevertheless, regardless of this choice, the space \(\puo{X}{\beta}\) for every ordinal \(\beta\) (even for \(\beta\) being successor ordinal) is a representative for the same element of \(\mathfrak{GH}\) space that~depends~only~on~\(X\) and \(\beta\).
\end{note}

Given a metric space \((X\), \(d)\), there is an interesting property of the canonical projections related to its iso--derivatives.

\begin{prop} \label{stabilize-so-iso}
Let \(\alpha < \beta\) be ordinals such that \(\puo{X}{\alpha}\) and \(\puo{X}{\beta}\) are isometric.
Then \(\pi_{\alpha}^{\beta}\) is an isometry.
Moreover, \(\pi_{\alpha}^{\alpha + 1}\) is an isometry.
\end{prop}

\begin{proof}
Of course, \(\pi_{\alpha}^{\beta}\) is an isometry as a nonexpansive mapping between isometric compact metric spaces.
But \(\pi_{\alpha}^{\beta} = \pi_{\alpha + 1}^{\beta} \circ \pi_{\alpha}^{\alpha + 1}\), so both \(\pi_{\alpha + 1}^{\beta}\) and \(\pi_{\alpha}^{\alpha + 1}\) must preserve distances.
\end{proof}

The immediate conclusion from this proposition is that for any compact metric space \((X\), \(d)\), its transfinite sequence of isometric derivatives \((X\), \(d)^{(\alpha)}\) must stabilize at some moment.
What is more, this moment is~a~countable ordinal.

\begin{prop} \(\) \label{prop before def of iht}
\begin{enumerate}
\item \label{countable heights}
For every compact metric space \((X\), \(d)\), there is a countable ordinal \(\alpha\) such that \(\pi_{\alpha}^{\alpha + 1}\) is an isometry.
\item \label{trivial Iso group}
If \((X\), \(d)\) and \(\alpha\) are as above, then for all ordinals \(\alpha \leqslant \beta \leqslant \gamma\), the canonical projection \(\pi_{\beta}^{\gamma}\) is~an~isometry.
In particular, \((\puo{X}{\alpha}\), \(\puo{d}{\alpha})\) is iso--rigid.
\end{enumerate}
\end{prop}

\begin{proof}
To prove (\ref{countable heights}), fix a compact metric space \((X\), \(d)\).
Let \(\Omega\) be the first uncountable ordinal number. \linebreak
Let \(\alpha_1\) be the~first~countable ordinal such that \(\puo{X}{\alpha_1} \in B_{d_{GH}}(\puo{X}{\Omega}\),~\(1)\).
Assume that we have already chosen countable ordinals \(\row{\alpha}{n}\) for~which~\(\puo{X}{\alpha_j} \in~B_{d_{GH}}(\puo{X}{\Omega}\),~\(\frac{1}{j})\) for~any~\(j \leqslant~n\).
Set \(\alpha_{n + 1}\) to~be the~smallest countable ordinal such~that~\(\alpha_{n + 1} > \alpha_{j}\) for \(j \leqslant n\) and \(\puo{X}{\alpha_{n + 1}} \in~B_{d_{GH}}(\puo{X}{\Omega}\),~\(\frac{1}{n + 1})\).
Note that in~this~situation \(\alpha \defeq \sup_{n \in \mathbb{N}}\alpha_n < \Omega\).
But \(\alpha\) is a limit ordinal, so, by the definition of iso--derivative and Theorem~\ref{main}, \(\puo{X}{\alpha_n} \longrightarrow \puo{X}{\alpha}\) in \(\mathfrak{GH}\).
However, by the definition of~\(\{\alpha_n\}_{n \in \mathbb{N}}\), \(\puo{X}{\alpha_n} \longrightarrow \puo{X}{\Omega}\) in \(\mathfrak{GH}\), and thus \(\puo{X}{\alpha}\) and \(\puo{X}{\Omega}\) are isometric.
Proposition~\ref{stabilize-so-iso} shows that \(\pi_{\alpha}^{\alpha + 1}\) is an isometry.

The point (\ref{trivial Iso group}) is a direct consequence of the previous proposition.
\end{proof}

The above propositions enable us to introduce a concept similar to that of the Cantor--Bendixson height (cf. \cite[page \(59\)]{Engelking}).

\begin{deff} \label{iso-height}
Let \((X\), \(d)\) be a compact metric space.
Define \emph{iso--height} of \((X\), \(d)\) to be the ordinal
\[
\displaystyle \OPN{iht}(X, \, d) \defeq \min\left\{\alpha \colon \ \pi_{\alpha}^{\alpha + 1}\textnormal{ is an isometry}\right\}.
\]
\end{deff}

The Proposition \ref{prop before def of iht} tells us that the transfinite sequence of countable iso--derivatives of any compact metric space must be eventually constant in the Gromov--Hausdorff world, so the above definition makes sense and \(\OPN{iht}\) is always a countable ordinal.
Furthermore, it turns out that any countable ordinal is iso--height of some compact metric space.

For the proof of this fact, the general property concerning the convergence in \(\mathfrak{GH}\) of a special kind of nets of disjoint unions of compact metric spaces will be useful.

\begin{prop} \label{conv of disjoint sums}
Let \(\{(X_\sigma\), \(d_\sigma)\}_{\sigma \in \Sigma}\) and \(\{(Y_\sigma\), \(\varrho_\sigma)\}_{\sigma \in \Sigma}\) be nets of compact metric spaces in \(\mathfrak{GH}\) converging to, respectively, \((X\), \(d)\) and \((Y\), \(\varrho)\).
For every \(\varepsilon > 0\) there exist \(\sigma_0 \in \Sigma\) and \(r > 0\) such that for every \(\sigma \in \Sigma\) satisfying \(\sigma \geqslant \sigma_0\), the~symmetric mapping \(\zeta_\sigma \colon (X_\sigma \sqcup Y_\sigma) \times (X_\sigma \sqcup Y_\sigma) \longrightarrow \mathbb{R}_{+}\) that extends the metrics \(d_\sigma\)~and~\(\varrho_\sigma\), and such that \(\zeta_\sigma(x_\sigma\), \(y_\sigma) = r\) for every \(x_\sigma \in X_\sigma\) and \(y_\sigma \in Y_\sigma\), is a metric.\\
Moreover, in the above, \(r\) can be chosen in such a way that
\[
\displaystyle (X_\sigma \sqcup Y_\sigma, \, \zeta_\sigma) \overset{\sigma \geqslant \sigma_0}{\longrightarrow} (X \sqcup Y, \, \zeta) \textnormal{ in } \mathfrak{GH},
\]
where \(\zeta\) is an admissible metric on \(X \sqcup Y\) with parameter \(r\).
\end{prop}

\begin{proof}
Fix \(\varepsilon > 0\) and find \(\sigma_0 \in \Sigma\) such that, for any \(\sigma \geqslant \sigma_0\), both the numbers \(d_{GH}(X_\sigma\),~\(X)\) and \(d_{GH}(Y_\sigma\),~\(Y)\) are smaller than \(\varepsilon\).
It follows that \(\sup_{\sigma \geqslant \sigma_0}\OPN{diam}_{d_\sigma}X_\sigma\) and \(\sup_{\sigma \geqslant \sigma_0}\OPN{diam}_{\varrho_\sigma}Y_\sigma\) are both not greater than \(\varepsilon + \max\{\OPN{diam}_{d}X\), \(\OPN{diam}_{\varrho}Y\}\).
To finish the proof of the first part of the proposition, it is enough to fix \(r \geqslant \frac{1}{2}(\varepsilon + \max\{\OPN{diam}_{d}X\), \(\OPN{diam}_{\varrho}Y\})\).

For the proof of the ``moreover'' part of the proposition, fix additionally \(\sigma \geqslant \sigma_0\).
By the definition of~the~Gromov--Hausdorff distance, one can find admissible metrics \(d_\sigma'\), \(\varrho_\sigma'\) on, respectively, \(X_\sigma \sqcup X\), \(Y_\sigma \sqcup~Y\) such that \((d_\sigma')_H(X_\sigma\), \(X)  < \varepsilon\), and \((\varrho_\sigma')_H(Y_\sigma\), \(Y) < \varepsilon\).
To finish the proof, it is enough to consider the~disjoint~union of these spaces with an admissible metric \(\zeta_\sigma'\) with parameter \(r\), and compute the Hausdorff distance between~the~appropriate pieces.
\end{proof}

\begin{thm} \label{all of countable heights}
For any countable ordinal \(\alpha\), there exists a compact metric space \((X\),~\(d)\) for which 
\[
\OPN{iht}(X, \, d) = \alpha.
\]
\end{thm}

\begin{proof}
We will prove it using transfinite induction.
It is clear that \(\OPN{iht}(\{0\}) = 0\) and \(\OPN{iht}(\{0\),~\(1\}) = 1\).
Assume that for an ordinal \(\alpha > 0\) one can find a compact metric space \((X\), \(d)\) with iso--height \(\alpha\).
Basing on~this, we~will produce a compact metric space whose iso--height is \(\alpha + 1\).
Denote by \((Y\), \(\varrho)\) a compact metric space being the disjoint union \(X \sqcup \puo{X}{\alpha}\) with an admissible metric with parameter 
\(r > \OPN{diam}_{d}X\).

We will prove that if \((Y\), \(\varrho)^{(\beta)}\) is the disjoint union \(\puo{X}{\beta} \sqcup \puo{X}{\alpha}\) with an admissible metric \(\puo{\varrho}{\beta}\) with~parameter~\(r\) (here it is important that \(\alpha > 0\)) for some \(\beta < \alpha\), then the space \((Y\), \(\varrho)^{(\beta + 1)}\) has the analogous property.
For this purpose, fix \(\varphi \in \OPN{Iso}(\puo{Y}{\beta}\), \(\puo{\varrho}{\beta})\), and \(\puo{x}{\beta} \in \puo{X}{\beta}\), \(\puo{x}{\alpha} \in \puo{X}{\alpha}\).
If \(\varphi(\puo{x}{\beta}) \in \puo{X}{\alpha}\), then by~the~choice of \(r\) we must have \(\varphi(\puo{x}{\alpha}) \in \puo{X}{\beta}\), and therefore \(\varphi(\puo{X}{\beta}) = \puo{X}{\alpha}\) and \(\varphi(\puo{X}{\alpha}) = \puo{X}{\beta}\), so \(\puo{X}{\beta}\) and \(\puo{X}{\alpha}\) are to be isometric.
However, this can only be the case when \(\alpha = \beta\).
But, as \(\alpha > \beta\), we must then have \(\varphi(\puo{x}{\beta}) \in \puo{X}{\beta}\) and \(\varphi(\puo{x}{\alpha}) \in \puo{X}{\alpha}\), and thus \(\varphi(\puo{X}{\beta}) = \puo{X}{\beta}\) and \(\varphi(\puo{X}{\alpha}) = \puo{X}{\alpha}\).
By~the~Proposition \ref{d1}, it is obvious that \(\puo{\varrho}{\beta + 1}([\puo{x}{\beta}]\), \([\puo{x}{\alpha}]) = r\).
Of~course, from the~above reasoning we also infer that \(\puo{\varrho}{\beta + 1}([\puo{x_1}{\beta}]\),~\([\puo{x_2}{\beta}]) =~\puo{d}{\beta + 1}([\puo{x_1}{\beta}]\),~\([\puo{x_2}{\beta}])\), with \(\puo{x_1}{\beta}\), \(\puo{x_2}{\beta} \in \puo{X}{\beta}\).
(We remember that, by Definition \ref{iso-height}, \(\puo{X}{\alpha + 1} = \puo{X}{\alpha}\) in~\(\mathfrak{GH}\).)
Therefore, we have indeed that \((Y\),~\(\varrho)^{(\beta + 1)}\) is the disjoint sum \(\puo{X}{\beta + 1} \sqcup \puo{X}{\alpha}\) with~an~admissible metric \(\puo{\varrho}{\beta + 1}\) with parameter \(r\).

We proved above that if \((Y\), \(\varrho)^{(\beta)}\) is the disjoint union \(\puo{X}{\beta} \sqcup \puo{X}{\alpha}\) with an admissible metric with~parameter \(r\),
then the analogous property holds for \((Y\), \(\varrho)^{(\beta + 1)}\).
To prove that \(\OPN{iht}(Y\), \(\varrho) = \alpha + 1\), it remains to~show that, if \(\beta \leqslant \alpha\) is a limit ordinal and \((Y\), \(\varrho)^{(\gamma)}\) is the disjoint union \(\puo{X}{\gamma} \sqcup \puo{X}{\alpha}\) with~an~admissible metric with parameter \(r\) for every ordinal \(\gamma < \beta\), the same remains true for the space \((Y\), \(\varrho)^{(\beta)}\).
But this~fact is~a~special case of Proposition~\ref{conv of disjoint sums}.

It remains to consider the limit case, i.e., that if \(\alpha\) is a limit ordinal and for every ordinal \(\beta < \alpha\) there~exists a compact metric space \((X_\beta\), \(d_\beta)\) with iso--height \(\beta\), then one can find a compact metric space \linebreak \((Y\), \(\varrho)\) with~iso--height \(\alpha\).
Fix a transfinite sequence \(\{(X_\beta\), \(d_\beta)\}_{\beta < \alpha}\) as above.
Let \(\{\alpha_n\}_{n \in \mathbb{N}}\) be a strictly increasing sequence of ordinals for which \(\sup_{n \in \mathbb{N}}\alpha_n = \alpha\).
Consider a sequence \(\{X_{\alpha_n}\}_{n \in \mathbb{N}}\).
For the convenience, we will write \((X_n\), \(d_n)\) instead of \((X_{\alpha_n}\), \(d_{\alpha_n})\) for \(n \in \mathbb{N}\).
Note that, without loss of~generality, we can assume that \(\{\OPN{diam}_{d_{n}}(X_{n})\}_{n \in \mathbb{N}}\) strictly decrease to \(0\) (if necessary, we can rescale all the~metrics, which~will~not change iso--heights of our spaces).
For the same reason, without loss of generality we can assume that there~exists a~sequence \(\{r_n\}_{n \in \mathbb{N}} \subset \mathbb{R}\) such that \(r_{n} > r_{n + 1} > \OPN{diam_{d_{n}}X_{n}}\) for every \(n \in \mathbb{N}\).
Let~\(X_\infty\) be a one--point metric space.
Consider a compact metric space \((Y\), \(\varrho)\) being the disjoint union \(\bigsqcup_{n \in \overline{\mathbb{N}}}X_n\) with~an~admissible metric~\(\varrho\) with~parameters \(\{r_n\}_{n \in \mathbb{N}}\).
We will prove that \(\OPN{iht}(Y\), \(\varrho) = \alpha\).

Fix an ordinal \(\beta < \alpha\) such that \((Y\), \(\varrho)^{(\beta)}\) is the disjoint union \(\bigsqcup_{n \in \overline{\mathbb{N}}}\puo{X_n}{\beta}\) with an admissible metric \(\puo{\varrho}{\beta}\) with parameters \(\{r_n\}_{n \in \mathbb{N}}\).
Similarly as before, we will first justify that for any \(\varphi \in \OPN{Iso}(\puo{Y}{\beta}\),~\(\puo{\varrho}{\beta})\) and every \(n \in \overline{\mathbb{N}}\) we have \(\varphi(\puo{X_n}{\beta}) = \puo{X_n}{\beta}\).
From this fact one can infer, like the previous time, that~\((Y\),~\(\varrho)^{(\beta + 1)}\) is the~disjoint union \(\bigsqcup_{n \in \overline{\mathbb{N}}}\puo{X_n}{\beta + 1}\) with an admissible metric \(\puo{\varrho}{\beta + 1}\) with parameters \(\{r_n\}_{n \in \mathbb{N}}\).
Choose arbitrary \(\varphi \in \OPN{Iso}(\puo{Y}{\beta}\),~\(\puo{\varrho}{\beta})\).
Note that the point that builds the space \(X_{\infty}\) is the only point of \(\puo{Y}{\beta}\) \linebreak whose all the spheres of radii from the set \(\{r_n\}_{n \in \mathbb{N}}\) are non--empty.
(Indeed, every point from \(\puo{X_n}{\beta}\) for \(n \in \mathbb{N}\) has empty sphere with radius \(r_{n + 1}\).)
This means that \(\varphi(\puo{X_{\infty}}{\beta}) = \puo{X_{\infty}}{\beta}\).
Because of that \(\varphi(\puo{X_n}{\beta}) = \puo{X_n}{\beta}\) for~\(n \in \mathbb{N}\), as~\(\varphi\) is surjective and preserves the aforementioned spheres.

The very last step of this proof is to show that if \(\beta \leqslant \alpha\) is a limit ordinal such that for every \(\gamma < \beta\) the~space \((Y\), \(\varrho)^{(\gamma)}\) is the disjoint union \(\bigsqcup_{n \in \overline{\mathbb{N}}} \puo{X}{\gamma}_{n}\) with an admissible metric with~parameters \(\{r_n\}_{n \in \mathbb{N}}\), then \((Y\),~\(\varrho)^{(\beta)}\) has similar property.
This will lead us to the conclusion that \((Y\), \(\varrho)^{(\alpha)}\) is the disjoint union \(\bigsqcup_{n \in \overline{\mathbb{N}}}\puo{X_n}{\alpha_n}\), where \(\alpha_{\infty} =~0\), with an admissible metric \(\puo{\varrho}{\alpha}\) with parameters \(\{r_n\}_{n \in \mathbb{N}}\).
To prove this~fact, we will proceed similarly as in~the~proof of~Proposition \ref{conv of disjoint sums}.
First, fix \(\varepsilon > 0\) and find \(n_0 \in \mathbb{N}\) such that \(r_n < \frac{\varepsilon}{2}\) for~every \(n \geqslant n_0\).
Let \(\gamma_0 < \beta\) be such that \(d_{GH}(\puo{X_j}{\gamma}\), \(\puo{X_j}{\beta}) < \frac{\varepsilon}{2}\) for all \(\gamma_0 \leqslant \gamma < \beta\) and \(j < n_0\).
Choose ordinal \(\gamma_0 \leqslant \gamma < \beta\).
Consider the~space \(Z\) being the disjoint union of the spaces \(\bigsqcup_{n = n_0}^{\infty}X_{n}^{(\gamma)}\)~and \(\bigsqcup_{n = n_0}^{\infty}X_{n}^{(\beta)}\), both considered with~admissible metrics with~parameters \(\{r_n\}_{n \in \mathbb{N}}\), and the space \(X_{\infty}\).
We~equip~\(Z\) with~an~admissible~metric~\(\zeta\) such~that \(\zeta(x_{m}^{(\gamma)}\), \(x_{n}^{(\beta)}) = r_m + r_n\) for~all \(x_{m}^{(\gamma)} \in X_{m}^{(\gamma)}\), \(x_{n}^{(\beta)} \in X_{n}^{(\beta)}\), and all \(m\), \(n \in \overline{\mathbb{N}}\) being at least \(n_0\).
For~\(j <~n_0\), let \(X_j'\) be the disjoint union \(X_j^{(\gamma)} \sqcup X_j^{(\beta)}\) with an admissible metric \(d_j'\) such~that \((d_j')_H(X_j^{(\gamma)}\),~\(X_j^{(\beta)}) < \frac{\varepsilon}{2}\).
Almost finally, to see that \(d_{GH}(\puo{Y}{\gamma}\), \(\bigsqcup_{n \in \overline{\mathbb{N}}} \puo{X}{\beta}_{n}) < \varepsilon\), for \(\bigsqcup_{n \in \overline{\mathbb{N}}} \puo{X}{\beta}_{n}\) equipped with the appropriate metric, it~is enough to~consider the space \(D\) being the disjoint union \(\bigsqcup_{j < n_0}X_j'\), considered with an admissible metric with parameters \(\row{r}{n_0 - 2}\), and \(Z\), with~an~admissible metric \(\delta\) such that \(\delta(x_j'\),~\(x_n') = r_j\) for~every \(x_j' \in X_j'\), where \(j < n_0\), every \(x_n' \in X_n'\), where \(n \in \overline{\mathbb{N}}\) is not smaller than~\(n_0\), while \(X_n' = X_n^{(\gamma)} \sqcup X_n^{(\beta)}\) for~\(n_0 \leqslant n < \infty\) and \(X_\infty' = X_\infty\).
It is easy to see that the way of defining the space \(D\) \linebreak guarantees that \(\delta_{H}(\varphi(\bigsqcup_{n \in \overline{\mathbb{N}}} \puo{X}{\gamma}_{n})\),~\(\psi(\bigsqcup_{n \in \overline{\mathbb{N}}} \puo{X}{\beta}_{n})) < \varepsilon\), where \(\varphi\), \(\psi\) are the obvious isometric embeddings of~\(\puo{Y}{\gamma} = \bigsqcup_{n \in \overline{\mathbb{N}}} \puo{X}{\gamma}_{n}\) and \(\bigsqcup_{n \in \overline{\mathbb{N}}} \puo{X}{\beta}_{n}\) into \(D\).
Finally, we proved the convergence \(\puo{Y}{\gamma} \longrightarrow \bigsqcup_{n \in \overline{\mathbb{N}}} \puo{X}{\beta}_{n}\) in \(\mathfrak{GH}\) when \(\gamma \longrightarrow \beta\), thus, since \((Y\), \(\varrho)^{(\gamma)} \longrightarrow (Y\), \(\varrho)^{(\beta)}\) in~\(\mathfrak{GH}\) when \(\gamma \longrightarrow \beta\) (see Corollary \ref{Convergence_of_a_net}), the~space \((Y\),~\(\varrho)^{(\beta)}\) has~the~desired form.
In particular, as announced earlier, this leads us to the conclusion that \((Y\), \(\varrho)^{(\alpha)}\) \linebreak is the disjoint union \(\bigsqcup_{n \in \overline{\mathbb{N}}}\puo{X_n}{\alpha_n}\), where \(\alpha_{\infty} =~0\), with an admissible metric \(\puo{\varrho}{\alpha}\) with parameters \(\{r_n\}_{n \in \mathbb{N}}\).

The proved above form of \((Y\), \(\varrho)^{(\alpha)}\) guarantees that starting from this point our transfinite sequence of~iso--derivetives of~\((Y\), \(\varrho)\) is constant.
Simultaneously, \((Y\), \(\varrho)^{(\beta)} \neq (Y\), \(\varrho)^{(\beta + 1)}\) in \(\mathfrak{GH}\) for every \(\beta < \alpha\), as~\(\OPN{iht} X_{\alpha_n} = \alpha_n\) for~every \(n \in \mathbb{N}\) and \(\sup_{n \in \mathbb{N}} \alpha_n = \alpha\).
Therefore, \(\OPN{iht}(Y\), \(\varrho) = \alpha\), what finishes the proof.
\end{proof}

\begin{ex} \label{images}
On the below figures one can see some examples of compact metric spaces together with their~iso--heights.
All these spaces are considered with Euclidian metric.
(The first one is a one--point space.)
\begin{figure}[H]
\begin{subfigure}{0.1475 \textwidth}
  \includegraphics[width = \linewidth]{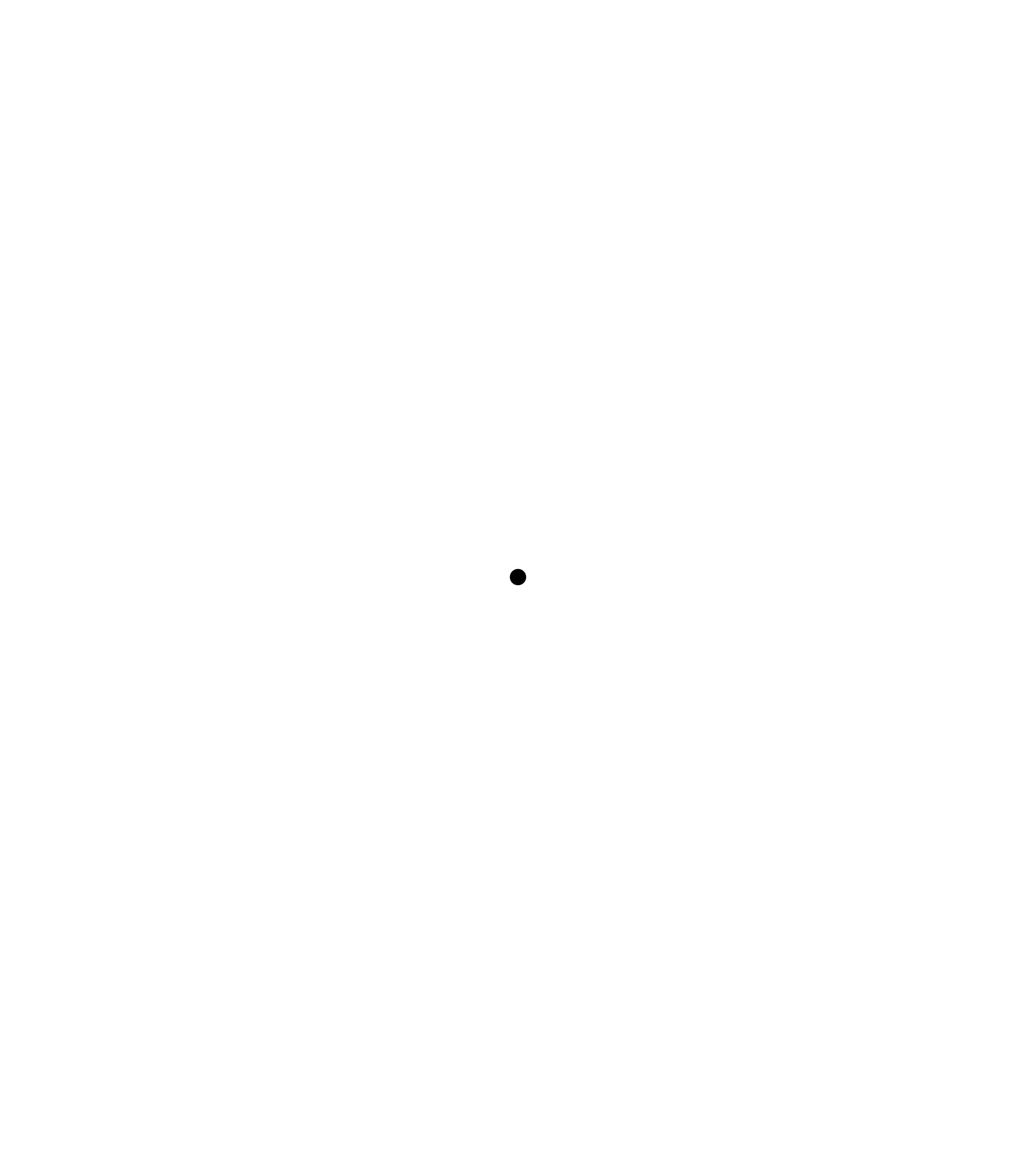}
  \caption{\(\OPN{iht} = 0\)}
  \label{0}
\end{subfigure} \hfil
\begin{subfigure}{0.1475 \textwidth}
  \includegraphics[width = \linewidth]{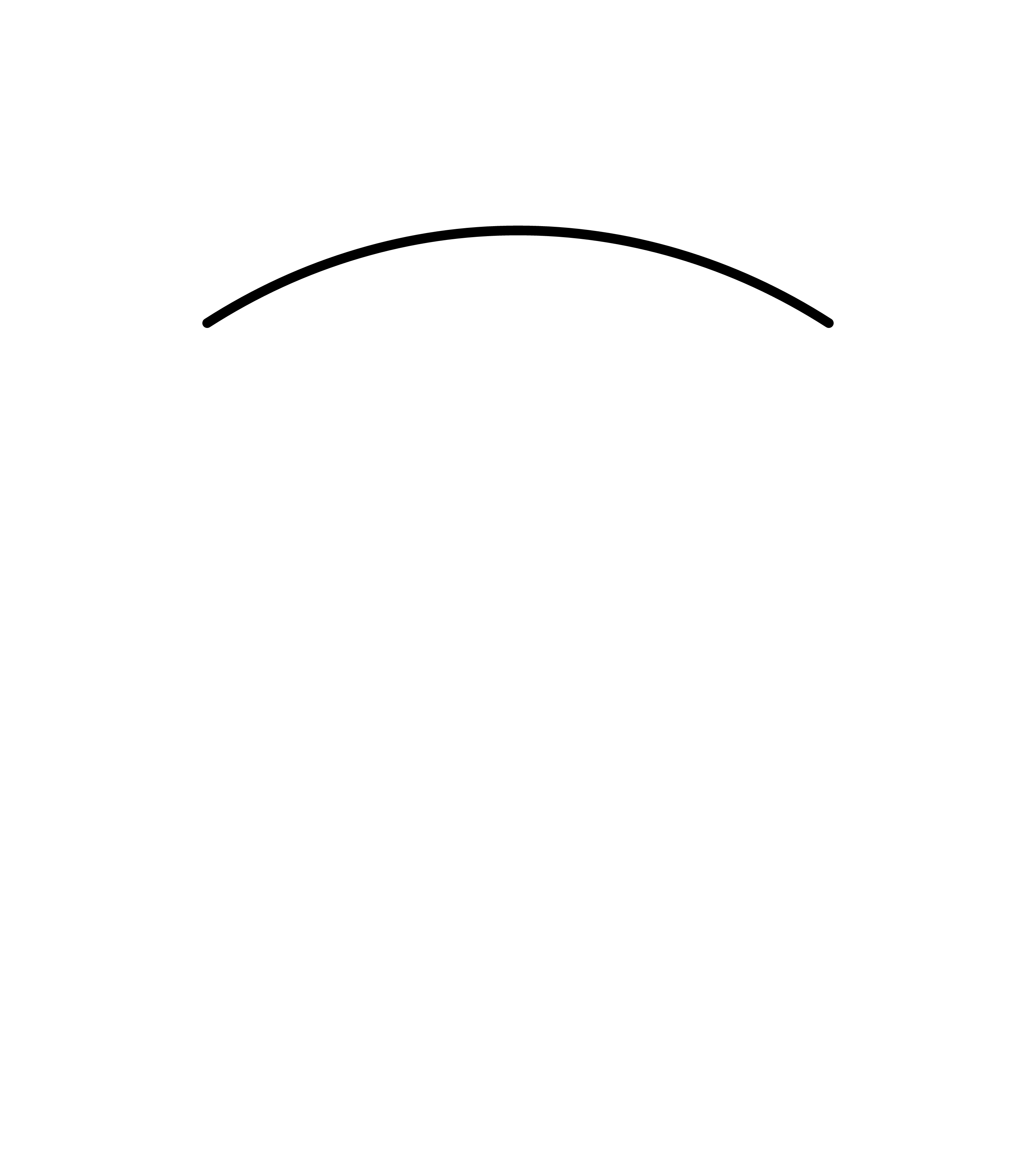}
  \caption{\(\OPN{iht} = \omega\)}
  \label{omega}
\end{subfigure}\hfil
\begin{subfigure}{0.1475 \textwidth}
  \includegraphics[width = \linewidth]{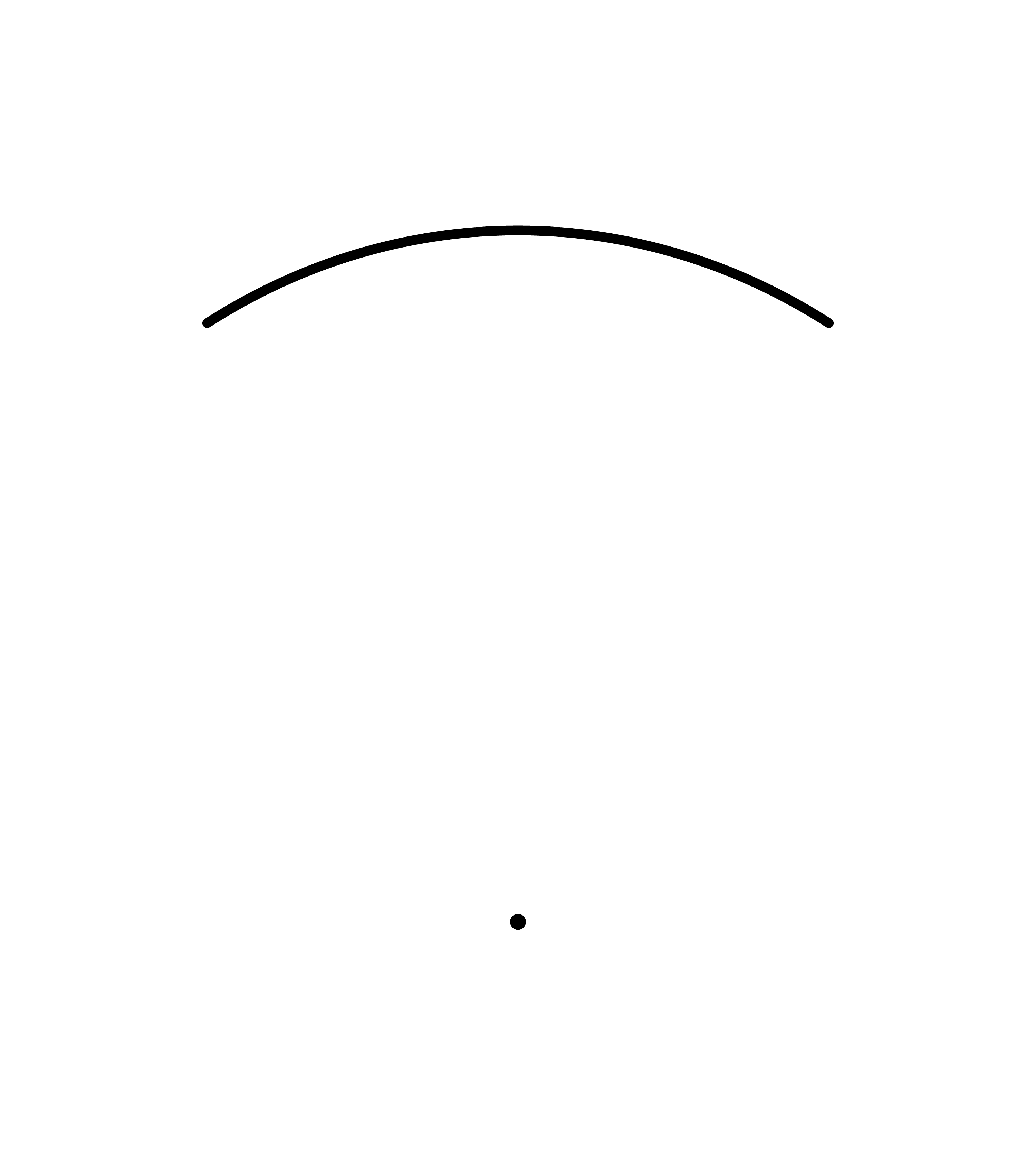}
  \caption{\(\OPN{iht} = \omega + 1\)}
  \label{omega + 1}
\end{subfigure}\hfil
\begin{subfigure}{0.1475 \textwidth}
  \includegraphics[width = \linewidth]{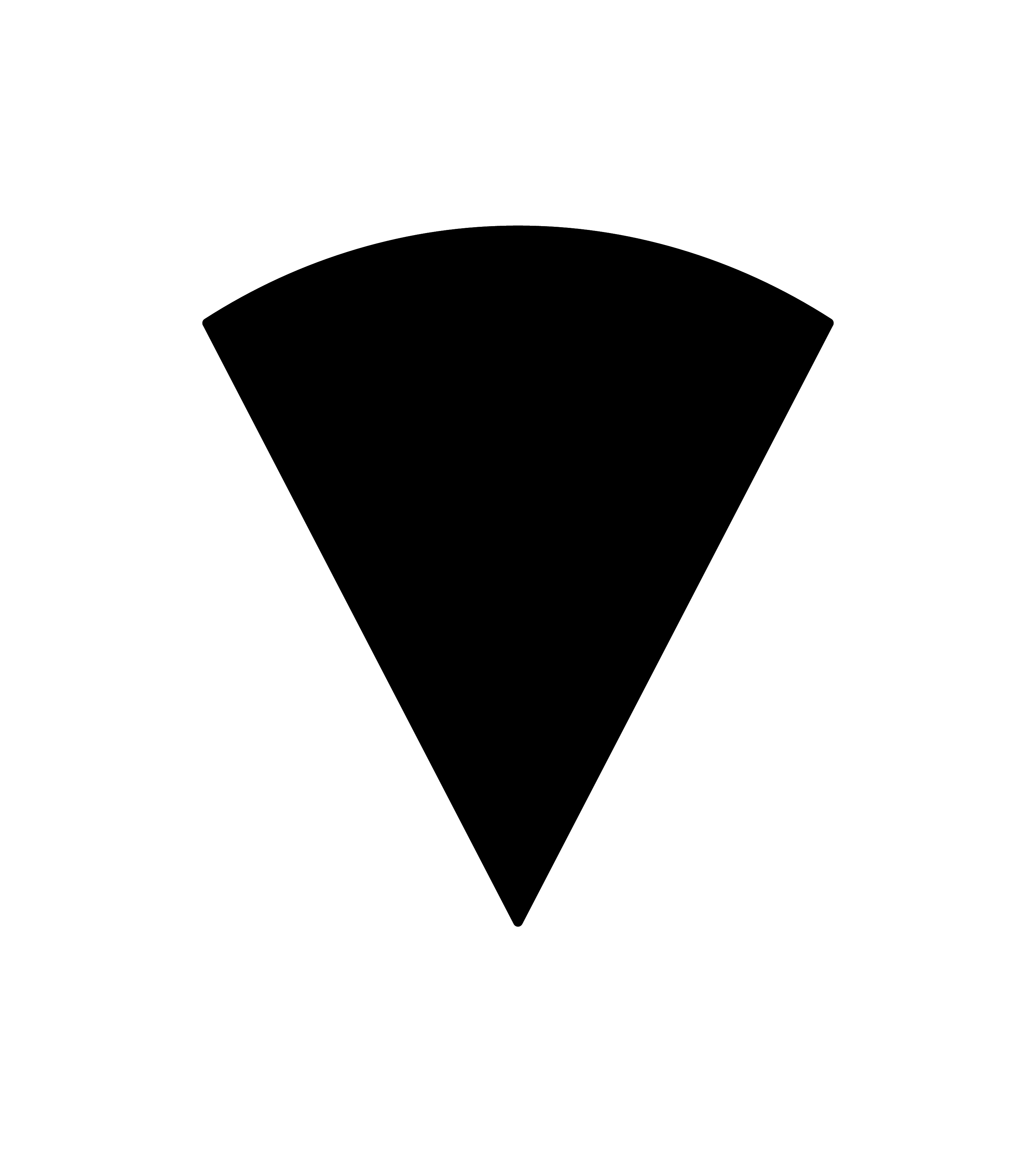}
  \caption{\(\OPN{iht} = 2\omega\)}
  \label{2omega}
\end{subfigure}
\end{figure}
\end{ex}

It is left as an exercise for the reader to find for every ordinal \(\alpha < \omega^2\) an explicit example of a compact metric space \((X\), \(d)\) such~that \(\OPN{iht}(X\), \(d) = \alpha\), and \((X\), \(d)^{(\alpha)}\) is a one--point compact metric space.

It is obvious that the above proof of Theorem \ref{all of countable heights} works for countable compact metric spaces as well. \linebreak
We can thus write the following conclusion.

\begin{cor}
For any countable ordinal number \(\alpha\), there exists a countable compact metric space \((X\), \(d)\) for~which \(\OPN{iht}(X\), \(d) = \alpha\).
\end{cor}

The next question regarding \(\OPN{iht}\) that arises naturally is whether it is possible to write the above corollary for~other special classes of compact sets.
For example, considering iso--rigid compact metric spaces (i.e.,~thanks to Definition \ref{iso-height}, the spaces with iso--heihgt \(0\)), it is interesting if, for every countable \(\alpha\), every such a space is an iso--derivative of~order \(\alpha\) for some compact metric space, at the same time, not appearing in the transfinite sequence of~isometric derivatives of that space before the \(\alpha\) derivative.
In particular, is this statement true for the one~point space?

In both cases, the argument for the successor ordinals from the proof of the Theorem \ref{all of countable heights} is still valid.
Certainly, more difficult to prove is the case of limit ordinals for which one should probably find completely different approach than this one presented above.
Intuition dictates that this case should has positive answer as well.

\begin{prob} \label{one point and stiff problem}
Is it true that for any countable ordinal \(\alpha\), and
\begin{enumerate}
\item every iso--rigid space, or
\item a one--point space \label{point for every hight}
\end{enumerate}
\((X\), \(d)\), there exists a compact metric space \((Y\), \(\varrho)\) such that \((Y\), \(\varrho)^{(\alpha)} = (X\), \(d)\) in \(\mathfrak{GH}\), and \(\alpha\) is the first ordinal with this property?
\end{prob}

The second case should be significantly easier to prove, for, during the proof, we can rescale our spaces to make them as small as we want. (We have already used this fact when proving the Theorem \ref{all of countable heights}.)

\begin{prop}
For every \(r > 0\), every countable ordinal \(\alpha\), and every compact metric space \((X\), \(d)\) we have \((X\), \(r \cdot d)^{(\alpha)} = (\puo{X}{\alpha}\), \(r \cdot \puo{d}{\alpha})\).
\end{prop}

\begin{proof}
It is enough to notice that \(\OPN{Iso}(X\), \(d) = \OPN{Iso}(X\), \(r \cdot d)\) and use formula for \(\puo{d}{1}\) from Proposition~\ref{d1}.
\end{proof}

One of the most general problems regarding the problem related to surjectivity of isometric derivative is~whether, treated as a mapping in Gromov--Hausdorff space, it is surjective.
If so, will the answer remain positive for iso--derivatives of order \(\alpha\) for any countable ordinal \(\alpha\)?

Finally, if the above is true, we can ask even more general question, being the counterpart of Problem~(\ref{one point and stiff problem}).

\begin{prob} \( \)
\begin{enumerate}
\item Is \(\puo{\Pi}{1} \colon (\mathfrak{GH}\), \(d_{GH}) \ni (X\), \(d) \longmapsto (X\), \(d)^{(1)} \in (\mathfrak{GH}\), \(d_{GH})\) surjective? \label{problem on surjectivity}
\end{enumerate}
If so, and \(\alpha\) is countable ordinal,
\begin{enumerate}
\item[(2)] is \(\puo{\Pi}{\alpha} \colon (\mathfrak{GH}\), \(d_{GH}) \ni (X\), \(d) \longmapsto (X\), \(d)^{(\alpha)} \in (\mathfrak{GH}\), \(d_{GH})\) surjective?
\end{enumerate}
Finally, if \((X\), \(d)\) is a compact metric space and \(\alpha\) is a countable ordinal, 
\begin{enumerate}
\item[(3)] is there a compact metric space \((Y\), \(\varrho)\) such that \((Y\), \(\varrho)^{(\alpha)} = (X\), \(d)\) in \(\mathfrak{GH}\), and \(\alpha\) is the first ordinal with this property?
\end{enumerate}
\end{prob}

Below we present some simple, but interesting tidbits on the subject of iso--derivative and iso--height of~a~compact metric space.

\begin{fa}
It is not true that if \((X\), \(d)\) can be isometrically embedded in \((Y\), \(\varrho)\), then \(\OPN{iht}(X\), \(d) \leqslant \OPN{iht}(Y\), \(\varrho)\).
\end{fa}

\begin{proof}
It is enough to consider so called Egyptian triangle (i.e., the triangle whose sides have the~lengths \(3\), \(4\) and \(5\)), equipped with the Euclidian metric.
Such a compact metric space has iso--height \(0\) and contains isometric copy of a closed interval with iso--height \(\omega\).
\end{proof}

Considering only the typical compact subspaces of Euclidean spaces, such as triangle, square, disc or torus, or the spaces (\subref{0}) -- (\subref{2omega}) presented in Example \ref{images}, one can get wrong impression that the isometric derivative \(\puo{X}{1}\) can be always isometrically embedded in \(X\).

\begin{ex}
Let \(T\) be the Egyptian triangle with the Euclidian metric.
Consider the disjoint union \linebreak \(X = T \sqcup T\) with an admissible metric \(d\) with parameter \(r = 2\frac{1}{2}\).
Then the iso--derivative \(\puo{X}{1}\) cannot~be isometrically embedded in \(X\).
What is more, there is no proper subspace of \(X\) that is homeomorphic to~\(\puo{X}{1}\).
\end{ex}

The last two facts will be about potential continuity of iso--derivative and iso--height.

\begin{fa}
It is not always true that if a net \(\{(X_{\sigma}\), \(d_{\sigma})\}_{\sigma \in \Sigma} \subset \mathfrak{GH}\) converges to a compact metric space \((X\),~\(d)\), then \(\OPN{iht}(X\), \(d) \leqslant \inf_{\sigma \in \Sigma} \OPN{iht}(X_{\sigma}\), \(d_{\sigma})\).
The same stands for \(\sup_{\sigma \in \Sigma} \OPN{iht}(X_{\sigma}\), \(d_{\sigma}) \leqslant \OPN{iht}(X\), \(d)\).
\end{fa}

\begin{proof}
(All the below spaces are considered with Euclidian metrics.)
It is obvious that one can find sequence of finite subsets of \([0\), \(1]\) with no symmetries that converges to~\([0\),~\(1]\) in the Hausdorff sense.
In this case, all~elements of the sequence has iso--height \(0\), while \(\OPN{iht}([0\), \(1]) = \omega\).
On the other hand, the sequence \(\{[0\),~\(\frac{1}{n}]\}_{n \in \mathbb{N}}\), whose all elements have iso--height \(\omega\), clearly converges in Hausdorff sense to \(\{0\}\), which has iso--height \(0\).
\end{proof}

\begin{fa}
Isometric derivative (as a mapping in \(\mathfrak{GH}\)) is not continuous.
What is more, the metric \(d_{\puo{GH}{1}}\) on \(\mathfrak{GH}\) given by
\[
\displaystyle d_{\puo{GH}{1}}(K, \, L) = d_{GH}(K, \, L) + d_{GH}(\puo{K}{1}, \, \puo{L}{1})
\]
is not complete.
\end{fa}

\begin{proof}
It is enough to consider the sequence of derivatives of the space in (\subref{omega + 1}) from Example \ref{images}.
\end{proof}

\section{Further Applications} \label{Section 6}
In the previous section, we were using the whole group \(\OPN{Iso}(X\), \(d)\), when defining iso--derivative of \((X\),~\(d)\).
However, nothing stands in the way of building the theory of isometric derivative or isometric height of~a~compact metric space in another way: by dividing by some proper subgroup of \(\OPN{Iso}(X\), \(d)\) that is defined by some set of rules that can be applied to every compact metric space, instead of the whole \(\OPN{Iso}(X\), \(d)\).
For example, we can divide by:
\begin{enumerate}[label=(\alph*), leftmargin = 3\parindent]
\item \(\OPN{Iso}_{\OPN{fin}}(X\), \(d) = \OPN{cl}(\langle \varphi \in \OPN{Iso}(X\), \(d) \colon \ \varphi \textnormal{ has finite rank} \rangle)\),
\item \(\OPN{Iso}_{\OPN{inv}}(X\), \(d) = \OPN{cl}(\langle \varphi \in \OPN{Iso}(X\), \(d) \colon \ \varphi \circ \varphi = \OPN{id}_X \rangle)\),
\item \label{version c} \(\OPN{Iso}_{\OPN{stab}}(X\), \(d) = \OPN{cl}(\langle \varphi \in \OPN{Iso}(X\), \(d) \colon \ \varphi(x) = x \textnormal{ on } G \rangle)\),
where \(G \subset X\) is a non--empty open set that~is~related in some sense to all \(\varphi \in \OPN{Iso}(X\), \(d)\),
\item \label{version d} \(\OPN{Iso}_{\OPN{fixed}}(X\), \(d) = \OPN{cl}(\langle \varphi \in \OPN{Iso}(X\), \(d) \colon \ \varphi(A) = A \rangle)\) for some non--empty closed set \(A \subset X\).
\end{enumerate}

In the above approaches to the problem of isometric derivative, one can define \(\puo{d}{1}\) on \(\puo{X}{1}\) in exactly the~same way as in Proposition \ref{d1}, and the rest of the steps also remains the same.

In the case of \ref{version c} and \ref{version d}, the bigger set \(G\) or \(A\) is, probably, in most cases, the more iso--rigid the space \((X\), \(d)\) should be.
For example, if in \ref{version d} we consider the interval \([0\), \(1]\) with fixed point \(p\), then \(\OPN{iht}_{\OPN{fixed}}(X\),~\(d\),~\(p)\) is equal to \(1\), when \(p = \frac{1}{2}\), and \(0\) in any other case, while the common isometric height of this interval is~\(\omega\). \linebreak
In these two cases one can investigate the properties of the subsets \(G\) and \(A\) of \(X\) for which the transfinite sequence of isometric derivatives stabilizes at, respectively, \(\OPN{cl}(G)\) and \(A\).

However, Theorem \ref{main} holds in a much more general context, thus one can construct the theory of~derivative in more general context than these considered previously.

Assume that with every compact metric space \((X\), \(d_X)\) from some category \(\mathfrak{cm}\) one can associate a mapping \(\alpha_X \colon X \times X \longrightarrow \mathbb{R}_{+}\) such that:
\begin{enumerate}
\item \(\alpha_Y \circ (u \times u) = \alpha_X\) for every isometric \(u \colon X \longrightarrow Y\), where \((Y\), \(d_Y) \in \mathsf{cm}\),
\item \(\alpha_X\) is symmetric and separates points,
\item \(\alpha_X \leqslant d_X\).
\end{enumerate}

Under these assumptions, for every compact metric spcace \((X\), \(d_X) \in \mathsf{cm}\) we can define its \(\alpha\)--derivative (of~order~\(1\)) as shown below.

Let \(\varrho_{\alpha, \, X}\) be the greatest pseudometric on \(X\) such that \(\varrho_{\alpha, \, X} \leqslant \alpha_X\), i.e.
\[
\displaystyle \varrho_{\alpha, \, X}(x, \, y) = \inf\left\{\Sigma_{k = 1}^{n} \alpha_X(z_{k - 1}, \, z_k) \colon \ x = z_0, \, \ldots, \, z_n = y \in X, \, n \in \mathbb{N}\right\}.
\]
Since \(\varrho_{\alpha, \, X} \leqslant d_X\), we can divide \(X\) by \(\{x \in X \colon \ \varrho_{\alpha, \, X}(x) = 0\}\) and get a compact metric space \((\puo{X}{1}_\alpha\), \(\puo{d}{1}_{\alpha, \, X})\) \linebreak such that the canonical mapping \(\pi \colon X \longrightarrow \puo{X}{1}_\alpha\) is nonexpansive.
Later, we can proceed similarly as in~the~section on the isometric derivative.

For example, one can consider a mapping
\[
\displaystyle \alpha_{\OPN{Iso}}(x, \, y) = \inf_{\varphi, \, \psi \, \in \, \OPN{Iso}(X, \, d)} d_X(\varphi(x), \, \psi(y))
\]
for \(x\), \(y \in X\).
As the mapping \(\alpha_{\OPN{Iso}}\) is as in Proposition \ref{d1}, we obtain similar situation as in the case of~isometric derivative.
Instead of isometries, in the above formula one can use homeomorphisms:
\[
\alpha_{\OPN{Homeo}}(x, \, y) = \inf_{g, \, h \, \in \, \OPN{Homeo}(X, \, d)} d_X(g(x), \, h(y))
\]
for \(x\), \(y \in X\).
Note that in this case there occur not homogeneous spaces \((X\), \(d_X)\) such that \((X\), \(\alpha_{\OPN{Homeo}})\) is homogeneous (e.g., an interval \([0\), \(1]\) with Euclidian metric).
We can also consider locally isometric homeomorphisms (i.e., these ones that preserve small distances).
More precisely, let \(G\) be the set of all \(u \colon X \longrightarrow X\) for which one can find \(\varepsilon > 0\) such that \(u \in \OPN{Iso}(X\), \(\min(d\), \(\varepsilon))\).
Then one may define
\[
\displaystyle \alpha_{G}(x, \, y) = \inf_{u, \, v \, \in \, G} d_X(u(x), \, v(y))
\]
for \(x\), \(y \in X\).
The last case we will mention is the following.
Fix \(M > 0\) and introduce a notation \linebreak \(\OPN{Lip}(X\), \(M) = \{u \in \OPN{Homeo}(X) \colon \ \OPN{Lip}(u) \leqslant M\), \(\OPN{Lip}(u^{-1}) \leqslant M\}\).
We can consider
\[
\displaystyle \alpha_{\OPN{Lip}}(x, \, y) = \inf_{u, \, v \, \in \, \OPN{Lip}(X, \, M)} d_X(u(x), \, v(y))
\]
for \(x\), \(y \in X\).

\end{document}